\documentclass[a4paper,10pt]{amsart}

\usepackage{a4}
\usepackage{graphics,graphicx}
\usepackage{amssymb,amsmath}
\usepackage[all]{xy}
\usepackage{stmaryrd}
\CompileMatrices
\usepackage[breaklinks]{hyperref}
\usepackage{enumitem}
\usepackage{xcomment} 

\usepackage{tikz}
\definecolor{light-light-gray}{gray}{0.95}
\definecolor{light-gray}{gray}{0.75}

\usetikzlibrary{patterns}
\usetikzlibrary{decorations.markings}
\usetikzlibrary{decorations.pathreplacing}
\usetikzlibrary{calc} 
\usepackage{graphicx} 

\newtheorem{theorem}{Theorem}[section]
\newtheorem{lemma}[theorem]{Lemma}
\newtheorem{proposition}[theorem]{Proposition}
\newtheorem{corollary}[theorem]{Corollary}
\newtheorem{conjecture}[theorem]{Conjecture}

\theoremstyle{definition}

\newtheorem{definition}[theorem]{Definition}
\newtheorem{example}[theorem]{Example}

\newtheorem{setting}[theorem]{Setting}
\newtheorem{construction}[theorem]{Construction}

\newtheorem{algo}[theorem]{Algorithm}
\newtheorem{remark}[theorem]{Remark}

\theoremstyle{remark}

\def\KK{\mathbb{K}}
\def\ZZ{\mathbb{Z}}

\def\QQ{\mathbb{Q}}

\def\sei{\mathrel{\mathop:}=}

\newcommand{\cone}[1]{\mathrm{cone}(#1)}

\renewcommand{\phi}{\varphi}

\def\KK{{\mathbb K}}

\def\ZZ{{\mathbb Z}}

\def\QQ{{\mathbb Q}}
\def\PP{{\mathbb P}}

\def\Ample{{\rm Ample}}
\def\BPF{{\rm BPF}}
\def\Bs{{\rm Bs}}
\def\B{{\rm \mathbf B}}

\def\SAmple{{\rm SAmple}}
\def\cov{{\rm cov}}
\def\rlv{{\rm rlv}}
\def\Pic{{\rm Pic}}
\def\Cl{{\rm Cl}}
\def\WDiv{{\rm WDiv}}

\def\Supp{{\rm Supp}}
\def\Bs{{\rm Bs}}
\def\Spec{{\rm Spec}}

\def\conv{{\rm conv}}
\def\cone{{\rm cone}}
\def\lin{{\rm lin}}

\def\tt#1{\texttt{#1}}
\def\out#1{\begingroup\tiny\begin{gather*} #1 \end{gather*}\endgroup}

\usepackage{pict2e,picture}

\makeatletter
\newcommand{\pnrelbar}{%
  \linethickness{\dimen2}%
  \sbox\z@{$\m@th\prec$}%
  \dimen@=1.1\ht\z@
  \begin{picture}(\dimen@,.4ex)
  \roundcap
  \put(0,.2ex){\line(1,0){\dimen@}}
  \put(\dimexpr 0.5\dimen@-.2ex\relax,0){\line(1,1){.4ex}}
  \end{picture}%
}
\newcommand{\precneq}{\mathrel{\vcenter{\hbox{\text{\prec@neq}}}}}
\newcommand{\prec@neq}{%
  \dimen2=\f@size\dimexpr.04pt\relax
  \oalign{%
    \noalign{\kern\dimexpr.2ex-.5\dimen2\relax}
    $\m@th\prec$\cr
    \noalign{\kern-.5\dimen2}
    \hidewidth\pnrelbar\hidewidth\cr
  }%
}
\makeatother

\begin{document}
\title[Algorithms for embedded monoids and base point free problems]
{Algorithms for embedded monoids and base point free problems}
\author[A.~Fahrner]{Anne Fahrner} 
\thanks{Supported by the Carl-Zeiss-Stiftung.}
\address{Mathematisches Institut, Universit\"at T\"ubingen,
Auf der Morgenstelle 10, 72076 T\"ubingen, Germany}
\email{fahrner@math.uni-tuebingen.de}

\begin{abstract}
We present algorithms for basic computations with 
monoids in finitely generated
abelian groups such as
monoid membership testing and computing an element
of the conductor ideal.
Applying them to Mori dream spaces,
we obtain algorithms to test whether a Weil divisor class of 
a given Mori dream space is base point free,
to compute generators of the monoid of base point
free Cartier divisor classes and to test whether 
a~$\QQ$-factorial
Mori dream space with known canonical class 
fulfills Fujita's base point free conjecture or not.
\end{abstract}

\maketitle

\section{Introduction}

A first part of this paper concerns \emph{embedded monoids}, 
that means finitely generated monoids in finitely generated
abelian groups, and thereby generalises ideas of the theory
on affine semigroups~\cite[Chapter 2]{brugu} to monoids with 
non-trivial torsion part.
We further present algorithms for embedded monoids, among others  
for computing generators of
intersections of embedded monoids and for computing an element
of the conductor ideal; see Algorithms~\ref{algo:inmon}~--~\ref{algo:pointcondid}.

In the second part of the paper, we apply these algorithms to base 
point free questions for Mori dream spaces.
Recall that Mori dream spaces, introduced by Hu and 
Keel~\cite{hukeel}, are characterized 
via their optimal behaviour with respect to the minimal model program. 
A particular interesting aspect of Mori dream spaces is 
their highly combinatorial structure~\cite{ADHL} -- in this regard 
they are a canonical generalisation of toric varieties.
Further well-known example classes are spherical varieties~\cite{BrKn},
smooth Fano varieties~\cite{BCHM} and all 
Calabi-Yau varieties of dimension at most three and
with polyhedral effective cone~\cite{McK}.
The combinatorial framework developed 
in~\cite{ADHL} allows algorithmic treatment of Mori dream spaces. 
Applying the aforementioned algorithms to Mori dream spaces,
we provide algorithms for testing whether a given Weil divisor class
is base point free and for computing
generators of the \emph{base point free monoid}, i.e.~the 
monoid of base point free Cartier divisor classes;
see Algorithms~\ref{algo:genBPF} and~\ref{algo:isbasepointfree}.

These algorithms, together with the non-emptyness of the 
conductor ideal of the base point free monoid,
play an important role in our main algorithm, 
Algorithm~\ref{algo:fujitabpf},
testing Fujita's base point free conjecture~\cite{fuconj}: 
this much studied conjecture claims
that for a smooth projective 
variety with canonical class~$\mathcal{K}_X$,
the Weil divisor class~$\mathcal{K}_X + m  \mathcal{L}$ is base 
point free for all ample
Cartier divisor classes~$\mathcal{L}$ and for all $m \geq \dim(X)+1$.
So far it is known to hold for smooth projective varieties 
up to dimension five~\cite{Re, EL1,Kaw,yezhu2}.
For toric varieties with arbitrary singularities,
Fujino~\cite{fujino} presented a proof of Fujita's
base point free conjecture.
Despite this substantial progress, Fujita's base point
free conjecture remains in general still open.
With Algorithm~\ref{algo:fujitabpf}, we provide a tool
for its algorithmic testing for~$\QQ$-factorial Mori dream spaces.
Since our algorithm makes use of the canonical class
$\mathcal{K}_X$, it applies to Mori dream
spaces with known $\mathcal{K}_X$.
This case appears quite often: for instance
if $X$ is spherical or
if its Cox ring is a complete intersection, see 
Remark~\ref{rem:knownKX} for details.

In~\cite{MonoidPackage}, we provide an implementation of our algorithms 
building on the two
{\tt{Maple}}-based software packages 
{\tt{convex}}~\cite{convex}
and
{\tt{MDSpackage}}~\cite{mdspackage}.
Using this implementation, we prove Fujita's base point free 
conjecture for a six-dimensional Mori dream space in
Example~\ref{ex:fujitabpf1},
and in Example~\ref{ex:fujitabpf2}, we study a locally factorial 
Mori dream space that does not fulfill
Fujita's base point free conjecture.
In addition, we study the more general question of the existence
of semiample Cartier divisor classes that are not base point free.
It is well-known that for Cartier divisors on complete toric varieties, 
semiampleness implies base point freeness.
For smooth rational projective varieties with a torus action 
of complexity one and Picard number two, the same statement
follows immediately 
from the classification done in~\cite{fahani}.
In Example~\ref{ex:smoothsemiamplenotbpf}, we present a 
first example of a smooth surface
of Picard number twelve admitting a semiample Cartier divisor with 
base points.

The author would like to thank J\"urgen Hausen for valuable 
discussions and comments. In addition, the author is grateful to the referees for their detailed and thorough review of the paper 
and for their highly appreciated comments and corrections.

\section{Embedded Monoids}

Let $K$ be a finitely generated abelian group.
We denote by $K = K^0 \oplus K^{\rm tor}$
the decomposition of $K$ into free and torsion part
and we write $K_\QQ := K \otimes_\ZZ \QQ$ for the 
associated rational vector space.
Note that each
$w\in K=K^0\oplus K^{\rm tor}$ can be
represented as $w=(w^0,w^{\rm tor})$
with unique elements $w^0\in K^0$ and
$w^{\rm tor} \in K^{\rm tor}$.
Every $w \in K$ defines an element~$w \otimes 1 \in K_\QQ$, 
which we denote as well by $w$ for short.
A \emph{cone} in a rational vector space always 
refers to a convex, polyhedral cone.
The relative interior of a cone $\tau \subseteq K_\QQ$ 
is denoted by $\tau^{\circ}$.

By an \emph{embedded monoid} we mean a pair 
$S \subseteq K$, where $S$ is a 
finitely generated submonoid of $K$.
For an embedded monoid $S \subseteq K$, we 
denote by 
$$
\cone(S) 
\ \sei \ 
\cone(w\otimes 1; \ w \in S) 
\ \subseteq \
K_\QQ
$$ 
the (convex, polyhedral) cone generated by 
the elements of $S$.
An embedded monoid $S \subseteq K$ 
is \emph{spanning} if $S$ generates $K$ 
as a group.
The \emph{saturation} of an embedded 
monoid $S \subseteq K$ is the embedded 
monoid
$$
\tilde S 
\ := \
\{w \in K; \; nw \in S \text{ for some } n \in \ZZ_{\ge 1}\}
\ \subseteq \ 
K\,.
$$ 
Let $S\subseteq K$ be an embedded monoid.
A non-empty set $M \subseteq K$ is called an \emph{$S$-module}
if~$S+M \subseteq M$ holds. We call an 
$S$-module $M$ \emph{ideal} if $M\subseteq S$ holds
and \emph{finitely generated} if there 
is a finite subset 
$\{m_1,\ldots,m_{\ell}\}\subseteq M$ with the property that
$M= \{m_1+s, \ldots,m_{\ell}+s; \ \, s\in S\}$ holds.

\begin{definition}\label{def:condid}
Let $S \subseteq K$ be an embedded monoid. 
The \emph{conductor ideal} of $S \subseteq K$ 
is the set
$$
c(\tilde S / S) 
\ \sei \
\{ x \in S; \ x+ \tilde S \subseteq S\}\,.
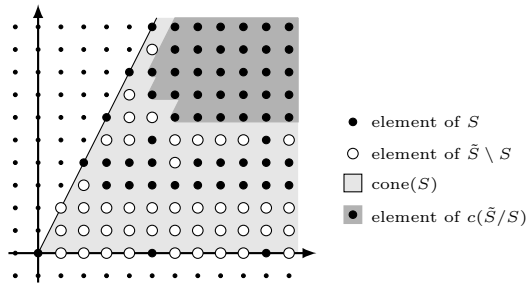
$$
\end{definition}

\goodbreak

\begin{figure}[ht]
\begin{tikzpicture}[scale=0.6]
    \path[fill=gray!20!] (0,0)--(5.7,0)--(5.7,5.2)--(2.6,5.2)--(0,0); 
    \draw  (0,0) -- (2.6,5.2); 
     \path[draw=gray!60!, fill=gray!60!]  (5.7, 5.2) -- (5.7,2.9) -- (3,2.9) -- (2.9,3) -- (3.1,3.4) -- (2.5,3.4) --(2.4, 3.5)-- (2.6,3.9) -- (2.5,3.9) --(2.4,4) -- (3,5.2) -- (5.7, 5.2); 
    \coordinate (Origin)   at (0,0);
    \coordinate (XAxisMin) at (-0.65,0);
    \coordinate (XAxisMax) at (6.1,0);
    \coordinate (YAxisMin) at (0,-0.65);
    \coordinate (YAxisMax) at (0,5.5);
    \draw [thick, -latex] (XAxisMin) -- (XAxisMax);
    \draw [thick, -latex] (YAxisMin) -- (YAxisMax);
    \foreach \x in {-1,0,1,2,...,11}{
      \foreach \y in {-1,0,1,2,...,10}{
        \node[draw,circle,inner sep=0.5pt,fill=black] at (0.5*\x,0.5*\y) {};
         }}
    \foreach \x in {0,1,2,...,11}{
      \foreach \y in {0}{
        \node[draw,circle,inner sep=1pt,fill=black] at (0.5*\x,0.5*\y) {};
         }}
    \foreach \x in {2,...,11}{
      \foreach \y in {3}{
        \node[draw,circle,inner sep=1pt,fill=black] at (0.5*\x,0.5*\y) {};
         }}
    \foreach \x in {2,...,11}{
      \foreach \y in {4}{
        \node[draw,circle,inner sep=1pt,fill=black] at (0.5*\x,0.5*\y) {};
         }}
    \foreach \x in {3,...,11}{
      \foreach \y in {5}{
        \node[draw,circle,inner sep=1pt,fill=black] at (0.5*\x,0.5*\y) {};
         }}
    \foreach \x in {3,...,11}{
      \foreach \y in {6}{
        \node[draw,circle,inner sep=1pt,fill=black] at (0.5*\x,0.5*\y) {};
         }}
    \foreach \x in {4,...,11}{
      \foreach \y in {7}{
        \node[draw,circle,inner sep=1pt,fill=black] at (0.5*\x,0.5*\y) {};
         }}
    \foreach \x in {4,...,11}{
      \foreach \y in {8}{
        \node[draw,circle,inner sep=1pt,fill=black] at (0.5*\x,0.5*\y) {};
         }}
    \foreach \x in {5,...,11}{
      \foreach \y in {9}{
        \node[draw,circle,inner sep=1pt,fill=black] at (0.5*\x,0.5*\y) {};
         }}
    \foreach \x in {5,...,11}{
      \foreach \y in {10}{
        \node[draw,circle,inner sep=1pt,fill=black] at (0.5*\x,0.5*\y) {};
         }}
    \foreach \x in {1,2,...,11}{
      \foreach \y in {1,2}{
        \node[draw,circle,inner sep=1.4pt,fill=white] at (0.5*\x,0.5*\y) {};
         }}   
    \foreach \x in {1,...,4}{
      \foreach \y in {0}{
        \node[draw,circle,inner sep=1.4pt,fill=white] at (0.5*\x,0.5*\y) {};
         }}  
    \foreach \x in {6,...,9,11}{
      \foreach \y in {0}{
        \node[draw,circle,inner sep=1.4pt,fill=white] at (0.5*\x,0.5*\y) {};
         }}  
    \foreach \x in {3,4,6,7,8,9,11}{
      \foreach \y in {5}{
        \node[draw,circle,inner sep=1.4pt,fill=white] at (0.5*\x,0.5*\y) {};
         }}  
    \foreach \x in {2}{
      \foreach \y in {3}{
        \node[draw,circle,inner sep=1.4pt,fill=white] at (0.5*\x,0.5*\y) {};
         }}
    \foreach \x in {6}{
      \foreach \y in {4}{
        \node[draw,circle,inner sep=1.4pt,fill=white] at (0.5*\x,0.5*\y) {};
         }}
    \foreach \x in {4}{
      \foreach \y in {6,7}{
        \node[draw,circle,inner sep=1.4pt,fill=white] at (0.5*\x,0.5*\y) {};
         }}
     \foreach \x in {5}{
      \foreach \y in {6,9}{
        \node[draw,circle,inner sep=1.4pt,fill=white] at (0.5*\x,0.5*\y) {};
         }}

\node[draw,circle,inner sep=1pt,fill=black] at (6.9,2.9) {};  
\node[right] at (7.1,2.9) {\tiny{element of $S$}}; 

\node[draw,circle,inner sep=1.4pt,fill=white] at (6.9,2.2) {}; 
\node[right] at (7.1,2.2) {\tiny{element of $\tilde{S}\setminus S$}};         

\node[draw, rectangle, fill=gray!20!] at (6.9,1.55) {}; 
\node[right] at (7.1,1.5) {\tiny{$\cone(S)$}};     

\node[draw, rectangle, draw= gray!60!, fill=gray!60!] at (6.9,.85) {}; 
\node[right] at (7.1,.8) {\tiny{element of $c(\tilde{S}/S)$}};  
\node[draw,circle,inner sep=1pt,fill=black] at (6.9,.85) {};  

  \end{tikzpicture}
\end{figure}

\goodbreak

\begin{lemma}\label{lem:tildeSfgmodule}
Let $S \subseteq K$ be an embedded 
monoid. Consider elements $x_1, \ldots, x_r \in S$ 
such that
$\{x_1 \otimes 1,\ldots,x_r \otimes 1 \}$
is a set of generators for~$\cone(S)$.
Then the finite set
$$
M
\ \sei \
\iota^{-1}\left(
\{ \sum_{i=1}^r\alpha_i (x_i \otimes 1); \ \;
 \alpha_i \in \QQ, \ 0\le\alpha_i\le 1 \}
\right)
\ \subseteq \ K \, ,
$$
where $\iota$ is the map
$\iota\colon K \to K\otimes\QQ,\ w\mapsto w\otimes1$,
generates~$\tilde S$ as an $S$-module.
In particular,~$\tilde{S}$ is a finitely 
generated~$S$-module.
\end{lemma}

\begin{proof}
In the case of a torsion free
group $K$, the statement on the finite 
generation of $\tilde S$ as an $S$-module
is Gordan's Lemma~\cite[Proposition 1.2.17]{CLS}.
The proof extends easily to the case 
of finitely generated abelian groups.
\end{proof}

\begin{proposition}\label{prop:condid}
Let $S \subseteq K$ be an embedded monoid. 
If $S\subseteq K$ is spanning, then the conductor ideal
$c(\tilde{S}/S)$ is non-empty, i.e.~it is in particular an $S$-module.
\end{proposition}
\begin{proof}
By definition, $S +c(\tilde S / S)  \subseteq c(\tilde S / S)$
holds, i.e.~we only have to show that 
$c(\tilde S / S) $ is non-empty.
In case of a torsion free group $K$, one can find 
a proof in~\cite[Proposition 2.33]{brugu}.
For finitely generated abelian groups
we may extend the proof as follows:
According to Lemma~\ref{lem:tildeSfgmodule},
we have
$\tilde S= \{m_1+s, \ldots,m_{\ell}+s; \ s\in S\}$
with some finite 
subset~$\{m_1,\ldots,m_{\ell}\} \subseteq \tilde S$.
By assumption, the embedded monoid $S\subseteq K$ 
is spanning. This yields
representations $m_i =x_i-y_i$ with 
$x_i, y_i \in S$. We claim that 
$z \sei \sum_{i=1}^\ell y_i$ is contained in the
conductor ideal $c(\tilde S / S)$. Indeed
$$
z+m_j = \sum_{
\substack{1\le i\le \ell\\
i\ne j}
}
y_i +x_j 
\ \in \
S
$$
holds for all $1\le j \le \ell$, i.e.
we have $z+\tilde S \subseteq S$.
\end{proof}

\begin{lemma}
\label{lem:intersection}
Let $K$ be a finitely generated abelian 
group and consider two subgroups~$K_1,K_2 \subseteq K$. 
Let $S_i \subseteq K_i$ be embedded monoids with
saturations $\tilde S_i$.
Then the following holds for the intersection
$S_{12} \sei S_1 \cap S_2$:
\begin{enumerate}
\item The intersection $S_1 \cap S_2 \subseteq K_1 \cap K_2$ 
is an embedded monoid.
\item
We have $\tilde S_{12}= 
\tilde S_1 \cap \tilde S_2$,
where $\tilde S_{12}$ denotes the
saturation of the embedded monoid
$S_{12} \subseteq K_1\cap K_2$.
\item
We have
$c(\tilde S_1/S_1)\cap c(\tilde S_2/S_2)
\subseteq 
c(\tilde S_{12}/S_{12})$.
\end{enumerate}
\end{lemma}

\begin{proof}
For~(i), only the finite generation of $S_1 \cap S_2$ 
needs some explanation, see for instance~\cite[Proposition~1.1.2.2]{ADHL}.
To prove the first inclusion of~(ii), let $x\in \tilde S_{12}$.
This means that we have~$x \in K_1\cap K_2$ and that there is 
$n\in\ZZ_{\ge 1}$ such that $nx \in S_1\cap S_2$ holds.
Clearly, this shows $x \in \tilde{S_1}\cap\tilde{S_2}$.
To prove the second inclusion, let 
$x \in \tilde{S_1}\cap\tilde{S_2}$. Hence~$x \in K_1\cap K_2$
holds and there are $n_1,\,n_2 \in \ZZ_{\ge 1}$ such that
$n_ix \in S_i, \, i=1,2$ hold. This means that
$n_1n_2x$ is contained in $S_1\cap S_2$, i.e.~we have 
$x \in \tilde{S}_{12}$. 
For~(iii), consider an element
$x\in c(\tilde S_1/S_1)\cap c(\tilde S_2/S_2)$.
This means that $x$ is contained in the intersection
$S_{12}$ and that
$x+\tilde S_i \subseteq S_i$ holds. 
With~(ii), we conclude that
$x+\tilde S_{12}$ is contained in $S_{12}$, 
i.e.~the conductor ideal 
of~$S_{12} \subseteq \lin_{\ZZ}(S_{12})$
contains $x$.
\end{proof}

Note that the following proposition is not true if we 
skip the condition that $\cone(S_1)^\circ$ and~$\cone(S_2)^\circ$
intersect non-trivially:
For instance, $\ZZ_{\ge 0}\subseteq\ZZ$ and~$\ZZ_{\le 0} \subseteq\ZZ$
define spanning embedded monoids, but the intersection 
$\{0\} \subseteq\ZZ$ is not spanning.

\begin{proposition}
\label{prop:intersection}
Let $K_1$ and~$K_2$ be subgroups of a finitely generated abelian 
group~$K$ and consider embedded monoids $S_i \subseteq K_i, \, i=1,2$.
If $\cone(S_1)^\circ \cap \cone(S_2)^\circ$ is non-empty
and $S_i \subseteq K_i$ is spanning for $i = 1,2$,
then $S_1 \cap S_2 \subseteq K_1 \cap K_2$ is a 
spanning embedded monoid.
\end{proposition}

\begin{proof}
We denote by~$S_{12}$ the intersection of~$S_1$ and~$S_2$. 
Note that $S_{12}\subseteq K_1\cap K_2$ 
is an embedded monoid by Lemma~\ref{lem:intersection}~(i).
Clearly, the group generated by~$S_{12}$ is contained in $K_1 \cap K_2$. It 
remains to show the opposite inclusion. 
We denote by~$\iota_1,\iota_2$
and~$\iota_{12}$ the maps defined by~$w\mapsto w\otimes 1$
fitting into the following diagramm:
$$
\begin{xy}
  \xymatrix @!C =1.15 cm @R=.5cm{
      K_1 \ar[rr]^{\hspace{-.6cm}\iota_1}   
      & &
      K_1\otimes\QQ  \ar@{}[r]|-*[@]{\supseteq}
      &
      \cone(S_1)
      \\
      K_1\cap K_2
      \ar@{}[u]|{\text{   \rotatebox{90}{$\subseteq$}    }}
      \ar[rr]^{\hspace{-.6cm}\iota_{12}} 
      \ar@{}[d]|{\text{   \rotatebox{90}{$\supseteq$}  }} 
      & &
      (K_1\cap K_2)\otimes\QQ
      \ar@{}[u]|{\text{   \rotatebox{90}{$\subseteq$}    }}
      \ar@{}[d]|{\text{   \rotatebox{90}{$\supseteq$}    }} 
      &
     \hspace{4cm} \tau\sei\cone(S_1)^\circ \cap \cone(S_2)^\circ
     \ar@{}[l]|-*[@]{\hspace{4cm}\supseteq}
      \\  
      K_2 \ar[rr]^{\hspace{-.6cm}\iota_2}             
      &  & 
      K_2\otimes\QQ 
      \ar@{}[r]|-*[@]{\supseteq}
      &
      \cone(S_2)\, .  
  }
\end{xy}
$$
Because of
$\tau\neq \emptyset$,
the rank of~$K_1\cap K_2$ and the dimension of~$\tau$ coincide.
Thus there are elements
$$
b_1,\ldots,b_r
\ \in \ 
\iota_{12}^{-1}(\tau) 
\ \subseteq \
\iota_1^{-1}(\cone(S_1)) \cap \iota_2^{-1}(\cone(S_2)) 
\ = \ 
\tilde{S_1}\cap\tilde{S_2}
$$
generating~$K_1\cap K_2$ as a group.
Furthermore $\tau \ne \emptyset$ implies that there is 
an element $x\in K_1\cap K_2$ such that
$x\otimes 1 \in \tau$ holds.
Recall that~$S_i\subseteq K_i$ are spanning monoids and thus
Proposition~\ref{prop:condid} shows that their conductor ideals
are non-empty. Since~$c(\tilde S_i/S_i)$ contains some
shifted copy of~$\tilde{S_i}$, there are some $m_i\in\ZZ_{\ge 1}$, $i=1,2$,
such that the integer multiple~$m_ix$ 
is contained in $c(\tilde S_i/S_i)$, $i=1,2$.
Hence there is some positive integer~$m\in\ZZ_{\ge 1}$
such that~$C$ contains the set of 
generators~$\{mx, mx+b_1,\ldots,mx+b_r\}$ 
for $K_1\cap K_2$.
It follows that
$$
K_1 \cap K_2 
\ = \
\lin_{\ZZ}(C)
\ \subseteq \
\lin_{\ZZ}(c(\tilde S_{12}/S_{12}))
\ \subseteq \
\lin_{\ZZ}(S_{12})
$$
holds, where the inclusion in the middle was shown in
Lemma~\ref{lem:intersection}~(iii)
and the inclusion on the right-hand side follows
since~$c(\tilde S_{12}/S_{12})$ is non-empty by the same Lemma
and thus contains some shifted copy of~$S_{12}$.
\end{proof}

In the following we describe some algorithms for monoids which, 
applied to Mori dream spaces, can be used for computing the base 
point free  monoid $\BPF(X)$, for testing whether a Cartier divisor
class is base point free and for computing a point of the 
conductor ideal of $\BPF(X) \subseteq \Pic(X)$.

\begin{algo}[inMonoid] \label{algo:inmon}
\textit{Input:} A finitely generated abelian group~$K^\prime$, 
generators~$s^\prime_1,\ldots,s^\prime_{t^\prime}\in K^\prime$
of an embedded monoid 
$S^\prime\sei\lin_{\ZZ_{\ge 0}}(s^\prime_1,\ldots,s^\prime_{t^\prime})\subseteq K^\prime$ and 
an element $w^\prime\in K^\prime$.\\
\textit{Output:} \textit{True} if $w^\prime$ is contained in $S^\prime$. 
Otherwise, \textit{false} is returned.
\begin{itemize}
\item 
By excluding the generators $s_i^\prime$ that equal~$0_K$, we achieve
a representation
$S^\prime\sei\lin_{\ZZ_{\ge 0}}(s^\prime_1,\ldots,s^\prime_t)$ 
with a natural number~$t\in\ZZ_{\le {t^\prime}}$ and with 
non-zero elements~$s_i^\prime$.
\item
We compute a canonical representation of the embedded monoid~$S^\prime\subseteq K^\prime$:
\begin{itemize}
\item Compute~$r,\tilde{r}\in\ZZ_{\ge 0}$ such that there is 
an isomorphism of groups
$\varphi\colon K^\prime \to K\sei \ZZ^r \oplus \bigoplus_{k=1}^{\tilde{r}} \ZZ/ a_i\ZZ$.
\item
Let~$S\sei\lin_{\ZZ_{\ge 0}}(s_1,\ldots, s_t)\subseteq K$,
where we set $s_i\sei\varphi(s_i^\prime)\in K$.
\item Set $w\sei\varphi(w^\prime)\in K$.
\end{itemize}
\item 
Let $Q \colon \ZZ^t\to K$ denote the homomorphism mapping 
$x=(x_1,\ldots,x_t)\in\ZZ^t$ to the integer combination $\sum x_i s_i$.
Denote by $Q^0$ the free part of $Q$, i.e.~with the projection 
$\pi\colon K \to K^0 = K / K^{\rm tor}$, we have $\pi\circ Q = Q^0$.
\item 
Compute the polyhedron $\mathcal{B} \, \sei \, (Q^0)^{-1}(w^0) \, \cap \, \QQ^t_{\ge 0}$.
\item If $\mathcal{B}$ is not bounded, then 
\begin{itemize}
\item for all $1\le i\le t$ do
\begin{itemize}
\item if $s_i^0 = 0_{K^0}$ holds, then let
$\mathcal{C}\sei\{1\le k\le \tilde{r}; \ \, s_{ir+k}\neq 0\}$
and
\vspace{-0.1cm}
$$ \mathcal{B} \ \sei \ \mathcal{B} \cap \{x \in \QQ^t; \; x_i \le \prod_{k \in \mathcal{C}} a_k \}\,.$$
\vspace{-0.2cm}
\end{itemize}
\end{itemize}
\item Compute the lattice points of the polytope $\mathcal{B}$, i.e.~compute
$B \sei \mathcal{B} \cap \ZZ^t$.
\item Return \textit{true} if there is a point $x\in B$ such that 
$Q(x)=w$ holds. Otherwise, return \textit{false}.
\end{itemize}
\end{algo}

\begin{proof}
We first show that in the end of the above algorithm, the polyhedron~$\mathcal{B}$ is a polytope.
Note that~$s_i\in K$ is a tupel $s_i=(s_{i1},\ldots, s_{ir}, s_{ir+1},\ldots , s_{ir+\tilde{r}})$ with integers $s_{ij}\in \ZZ, \, 1\le j\le r$,
and elements~$s_{ir+k}\in\ZZ/a_k\ZZ, \, 1\le k\le \tilde{r}$.
Via an isomorphism of abelian groups $K\to K$ we may assume that
$\cone(S)$ is contained in~$\QQ^r_{\ge 0}$, 
i.e.~we have~$s_{i1},\ldots, s_{ir} \ge 0$ for all~$1\le i\le t$.
Consider the polyhedron
$$\mathcal{A} \sei (Q^0)^{-1}(w^0) \, \cap \, \QQ^t_{\ge 0} \, .$$
Note that $\mathcal{A}$ contains exactly those lattice 
points $x=(x_1,\ldots,x_t)\in \ZZ^t_{\ge 0}$ with the property that
$$
\sum_{i=1}^t x_i (s_{i1}, \ldots, s_{ir}) \ = \ \sum_{i=1}^t x_i s_i^0 \ =\  Q^0(x) \ = \ w^0 \ = \ (w_1,\ldots, w_r) 
$$
holds. This means that the integer coefficient~$x_i$ is smaller 
than~$\lfloor\frac{w_j}{s_{ij}}\rfloor$
for all~$1\le j\le r$ with $s_{ij}\neq 0$, 
where~$\lfloor\cdot\rfloor$ denotes the floor function.
In particular, we have
$$
\mathcal{A}
\ \subseteq \
\left\lbrace
x \in \QQ^t_{\ge 0}; \ \ x_i
\le 
\min\left( 
\lfloor\frac{w_j}{s_{ij}}\rfloor;\;\, 1\le j\le r, \, s_{ij}\neq 0 
\right)
\right\rbrace
$$
for all $1\le i\le t$ such that $s_i^0\neq 0_{K^0}$ holds, i.e.
$\mathcal{A}$ is bounded with respect to these coordinate directions~$i$.
For all other coordinate directions $1\le i\le t$ of $\ZZ^t$, 
i.e.~of those with $s_i^0= 0_{K^0}$,
the above algorithm computes a bound~$b_i$, where  
$$
b_i
\ \sei \ 
\prod_{k \in \mathcal{C}} a_k 
\ \in \ \ZZ \, , 
\qquad
\mathcal{C}\sei\{1\le k\le \tilde{r}; \ \, s_{i r+k}\neq 0_{\ZZ/a_k\ZZ}\} \,.
$$
Note that~$\mathcal{C}$ is non-empty since in the first step of the algorithm,
we excluded the~$s_i^\prime$ that are zero.
We conclude 
that 
$$
\mathcal{B} = \mathcal{A} \cap \{x\in\QQ^t; \; x_i \le b_i 
\text{ for all } 1\le i\le t \text{ with } s_i^0= 0_{K^0} \}
$$ 
is indeed a polytope and thus~$B=\mathcal{B}\cap \ZZ^t$ is a finite set.

We now explain why the above algorithm has the claimed output.
We need to show that $w^\prime \in S^\prime$ holds if and only if the algorithm
returns true.
Clearly, $w^\prime \in S^\prime$ holds if and only if 
$w$ is contained in~$S$.
This in turn is the case if and only if 
there is an elment
$x\in \mathcal{A} \cap \ZZ^t_{\ge 0}$ 
such that $Q(x)=w$ holds.
If $\mathcal{A}$ is a polytope, there is nothing to show.
If $\mathcal{A}$ is unbounded we showed above that there is an 
index~$1\le i\le t$ 
such that $s_i^0=0_{K^0}$ holds.
It remains to show that 
the following assertions are equivalent:
\begin{enumerate}
\item 
There is an element $x\in\mathcal{A}\cap\ZZ^t_{\ge 0}$ 
such that $Q(x)=w$ holds.
\item 
There is an 
element 
$y\in B_i 
\sei 
\mathcal{A} \cap \{x\in\ZZ^t_{\ge 0}; \; x_i \le b_i\}$ 
with~$Q(y)=w$.
\end{enumerate}
Since~$B_i \subseteq \mathcal{A}\cap\ZZ^t_{\ge 0}$ holds,
the direction ``(ii)$\Rightarrow$(i)'' is obvious.
For the other direction, recall that~$b_i$ is the product 
of all~$a_k$, $1\le k\le \tilde r$, 
with~$s_{i r+k}\neq 0_{\ZZ/a_k\ZZ}$.
Since~$s_i^0= 0 _{K^0}$ holds, 
we thus obtain $\alpha s_i = \alpha^\prime s_i$ for all 
integers $\alpha,\, \alpha^\prime$ with
$\alpha\equiv \alpha^\prime (\text{mod}~b_i)$.
This means that it is sufficient to look at coefficient vectors $x\in\ZZ^t_{\ge 0}$
with~$x_i\le b_i$, i.e.~(i) implies~(ii).
As argued above, this completes the proof.
\end{proof}

\begin{example}\label{ex:1}
Consider the abelian group~$K\sei\ZZ\oplus\ZZ/4\ZZ$, its 
elements~$s_1\sei (0,\bar{2})$, $s_2\sei (1,\bar{1})$, $s_3\sei (3,\bar{2})$,
$w\sei (3,\bar{1})$ and the monoid $S\sei\lin_{\ZZ_{\ge 0}}(s_1,s_2,s_3)$
depicted in the picture below.
Algorithm~\ref{algo:inmon} applied to $S$
and to~$w$ does the following:
\begin{itemize}
\item The map~$Q$ is defined 
by~$\ZZ^3 \to K, \, 
(x_1,x_2,x_3)\mapsto 
(x_2+3x_3, \alpha)$,
where we set $\alpha\sei((2x_1+x_2+2x_3)+4\ZZ) \in \ZZ/4\ZZ$.
Its free part~$Q^0$ is given by~$\ZZ^3 \to \ZZ, \, (x_1,x_2,x_3)\mapsto x_2+3x_3$.
\item The polyhedron $(\QQ^0)^{-1}(w^0)$ is given by~$\QQ\times \{(3-3\beta, \, \beta)\, ;\; \beta\in \QQ\}$. 
Thus the algorithm starts with the polyhedron 
$$
\mathcal{B} 
\ = \ \QQ_{\ge 0}\times \{(3-3\beta, \, \beta)\, ;\; \beta\in \QQ, \, 0\le \beta\le 1\}.
$$
\item Since~$\mathcal{B}$ is unbounded and $s_i^0$ is zero if and only $i=1$ holds,
the algorithm then computes the polytope
$$
\mathcal{B} \sei \mathcal{B} \cap \{ x\in\QQ^t; \; x_i \le 4\}.
$$ 
Now we have $\mathcal{B} = \{(\alpha,\, 3-3\beta, \,\beta)\,;\; \alpha,\beta\in\QQ, \, 0\le \alpha\le 4, \, 0\le \beta\le 1\}$.
\item In a next step, the algorithm computes the lattice points~$B$ of~$\mathcal{B}$:
$$
B \ = \ \{(\alpha, 3, 0), \, (\alpha, 0, 1); \; \alpha\in\ZZ, \, 0\le\alpha\le 4\}.
$$
\item
Since $Q((1,3,0))=1s_1+3s_2+0s_3=w$ holds, the algorithm returns \textit{true}.
\end{itemize}
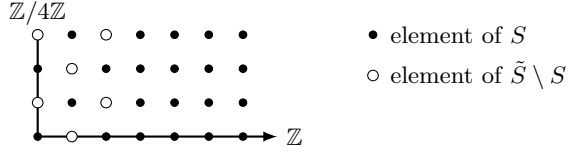
\begin{figure}[ht]
\begin{tikzpicture}[scale=0.9]
    
    \coordinate (Origin)   at (0,0);
    \coordinate (XAxisMin) at (0,0);
    \coordinate (XAxisMax) at (3.5,0);
    \coordinate (YAxisMin) at (0,0);
    \coordinate (YAxisMax) at (0,1.5);
    \draw [thick, -latex] (XAxisMin) -- (XAxisMax);
    \draw [thick] (YAxisMin) -- (YAxisMax);
    \foreach \x in {0,2,3,...,6}{
        \node[draw,circle,inner sep=1pt,fill=black] at (0.5*\x,0) {};
         }
    \foreach \x in {1,3,4,...,6}{
        \node[draw,circle,inner sep=1pt,fill=black] at (0.5*\x,.5) {};
         }
    \foreach \x in {0,2,3,...,6}{
        \node[draw,circle,inner sep=1pt,fill=black] at (0.5*\x,1) {};
         }
    \foreach \x in {1,3,4,...,6}{
        \node[draw,circle,inner sep=1pt,fill=black] at (0.5*\x,1.5) {};
         }         
    \node[draw,circle,inner sep=1.4pt,fill=white] at (0,0.5) {};
    \node[draw,circle,inner sep=1.4pt,fill=white] at (0,1.5) {};    
    \node[draw,circle,inner sep=1.4pt,fill=white] at (0.5,0) {};
    \node[draw,circle,inner sep=1.4pt,fill=white] at (0.5,1) {};
    \node[draw,circle,inner sep=1.4pt,fill=white] at (1,0.5) {};
    \node[draw,circle,inner sep=1.4pt,fill=white] at (1,1.5) {};

\node[right] at (3.5,0) {\small{$\ZZ$}};  
\node[above] at (0,1.5) {\small{$\ZZ/4\ZZ$}};

\node[draw,circle,inner sep=1pt,fill=black] at (4.9,1.5) {};  
\node[right] at (4.9,1.5){\small{~element of $S$}}; 

\node[draw,circle,inner sep=1.4pt,fill=white] at (4.9,.9) {}; 
\node[right] at (4.9,.9) {\small{~element of $\tilde{S}\setminus S$}}; 

  \end{tikzpicture}
\end{figure}
\end{example}

\begin{algo}[generatorsIntMonoid] \label{algo:genintmon} 
\textit{Input:} Two subgroups~$K_1, \, K_2$ of a finitely generated 
abelian group~$K$
and generators~$s_{i1},\ldots,s_{in_i} \in K_i$
of embedded monoids $S_i \sei \lin_{\ZZ_{\ge 0}} (s_{i1},\ldots,s_{in_i}) \subseteq K_i, \, i=1,2$.\\
\textit{Output:} A set of generators for the embedded monoid
$S_1\cap S_2 \subseteq K_1\cap K_2$.
\end{algo}

\begin{itemize}
\item 
Let $\varphi\sei\varphi_1\times\varphi_2\colon
\ZZ^{n_1+n_2}\to K \times K$ be the 
homomorphism of abelian groups defined through 
 $\varphi_i\colon\ZZ^{n_i}\to K, e_{ij} \mapsto s_{ij}$, where the
$e_{ij}$ denote the canonical base vectors of $\ZZ^{n_i}$.
Furthermore, define the projection 
$\psi\colon K \times K \to (K \times K)/\Delta$, 
where $\Delta \sei \{(k,k); \; k \in K\}$ denotes the diagonal.
\item
Compute the kernel of 
$\beta\sei\psi\circ\varphi$.
\item Consider the isomorphism of abelian groups
$\iota \colon \ZZ^r \to \ker(\beta)$ and compute generators
$g_1,\ldots,g_t$ for $\ZZ^r \cap \iota^{-1}(\QQ^{n_1+n_2}_{\ge 0})$.
\item 
Define the projection 
$\pi \colon K \times K \to K, \; (x,y) \mapsto x$  on the 
first factor and return the set 
$\{ (\pi\circ\varphi\circ\iota)(g_j); \; j=1,\ldots,t\}$.
\end{itemize}

\begin{proof}
According to Gordan's lemma~\cite[Proposition 1.2.17]{CLS},
there are generators $g_1,\ldots g_t$
for the monoid $\ZZ^r \cap \iota^{-1}(\QQ^{n_1+n_2}_{\ge 0})$.
Let $M \sei \ker(\beta) \cap \ZZ^{n_1+n_2}_{\ge 0}$ and consider the diagramm
$$ 
\xymatrix{
\ZZ^r \cap \iota^{-1}(\QQ^{n_1+n_2}_{\ge 0})
\ar[r]
\ar@{}[d]|{ \rotatebox[origin=c]{270}{\small{$\subseteq$}}} 
&
M \!\!\!\!\!\!\!\!
\ar@{}[r]|{ \subseteq} 
\ar@{}[d]|{ \;\;\;\;\;\rotatebox[origin=c]{270}{\small{$\subseteq$}}} 
&
\ZZ^{n_1+n_2}_{\ge 0}
\ar[r]
\ar@{}[d]|{ \rotatebox[origin=c]{270}{\small{$\subseteq$}}} 
&
S_1 \times S_2
\ar@{}[d]|{ \rotatebox[origin=c]{270}{\small{$\subseteq$}}} 
\ar[r]
&
(K \times K) / \Delta\ar@{}[d]|{ \rotatebox[origin=c]{270}{$=$}} \\
\ZZ^r
\ar[r]^{\iota}_{\cong}
&
\ker(\beta) \!\!\!\!\!\!\!\!
\ar@{}[r]|{ \subseteq} 
&
\ZZ^{n_1+n_2}
\ar[r]^{\varphi}
\ar@/_1.5pc/[rr]_\beta
&
K \times K
\ar[r]^{\psi \; \; \; \; \; \;}
&
(K \times K) / \Delta\,.
}
$$
With the projection 
$\pi \colon K \times K \to K, \; (x,y) \mapsto x$  on the 
first factor, we obtain
$$
(\pi\circ\varphi\circ\iota)\left(   
\ZZ^r \cap \iota^{-1}(\QQ^{n_1+n_2}_{\ge 0}  )   \right)
\ = \
(\pi\circ\varphi)( M  )
\ = \
S_1 \cap S_2\, ,
$$
where the last equality is true since
$\varphi(M)=\{(a,b) \in S_1\times S_2; \; a=b \}$ holds.
We conclude that $\{ (\pi\circ\varphi\circ\iota)(g_j); \; j=1,\ldots,t\}$
is a set of generators for $S_1\cap S_2$. 
\end{proof}

\begin{example}
Consider the abelian group~$K_1\sei K_2\sei K\sei\ZZ$ as well as its 
elements~$s_{11}\sei 2$, $s_{12}\sei 5$, and~$s_{21}\sei 3$.
Algorithm~\ref{algo:genintmon} applied to the monoids
$S_1\sei\lin_{\ZZ_{\ge 0}}(s_{11},s_{12})$ and~$S_2\sei\lin_{\ZZ_\ge 0}(s_{21})$
depicted in the figure below 
proceeds as follows:
\begin{itemize}
\item The map~$\varphi$ is given by~$\ZZ^3 \to \ZZ\times\ZZ, \, e_{11}\mapsto s_{11}, \, e_{12}\mapsto s_{12},\, e_{21}\mapsto s_{21}$,
where $3=2+1=n_1+n_2$ holds. 
To be precise,~$\varphi$ is defined by the matrix
$$
\left(
\begin{array}{ccc}
2 & 5 & 0\\
0 & 0 & 3
\end{array}
\right) \, .
$$
\item
The kernel of~$\beta$
is given by
$
\ker(\beta) 
 = 
\lin_{ \ZZ} 
\big( (1,2,4),(0,3,5) \big)
  \cong 
 \ZZ^2  .
$
\item 
The isomorphism $\iota\colon\ZZ^2 \to \ker(\beta)$ is defined
by mapping the first canonical base vector of~$\ZZ^2$ to~$(1,2,4)$
and the second one to~$(0,3,5)$.
\item 
We have 
$\QQ^2 \cap \iota^{-1}(\QQ^{3}_{\ge 0}) = \cone((3,-2),(0,1))$.
According to Gordan's Lemma, 
computing the lattice points of the polytope
$$\conv((0,0),(3,-2),(0,1),(3,-1))$$ gives the following generators
for the monoid $\ZZ^2 \cap \iota^{-1}(\QQ^{3}_{\ge 0})$:
$$
(0,0), \, (0,1), \, (1,0), \, (2,-1),\, (3,-2),\, (3,-1)\, .
$$
\item 
Applying $\pi\circ\varphi\circ\iota$ to those generators gives
the generators~$0, 15, 12, 9, 6, 21$ for $S_1\cap S_2$.
Note that this list is not a hilbert basis. 
To speed up the computation process in~\cite{MonoidPackage}, 
some reduction mechanisms were implemented.
\end{itemize}
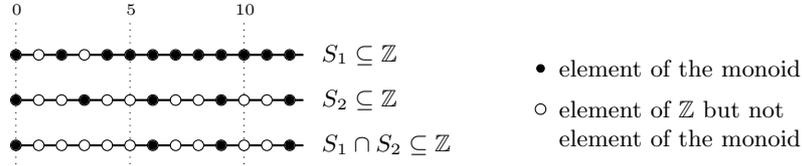
\begin{center}
\begin{figure}[ht]
\begin{tikzpicture}[scale=0.6]

    \foreach \x in {0,5,10}{
        \node[] at (0.5*\x-.05,1) {\tiny{ \x}};
        \draw[thin, dotted]  (0.5*\x, .7) -- (0.5*\x,-2.5);         
     }    
    \coordinate (Origin)   at (0,0);
    \coordinate (XAxisMin) at (-0.1,0);
    \coordinate (XAxisMax) at (6.3,0);
    \draw [thick] (XAxisMin) -- (XAxisMax);
    \foreach \x in {0,1,...,12}{
        \node[draw,circle,inner sep=1.4pt,fill=white] at (0.5*\x,0) {};
         }
    \foreach \x in {0,2,4,5,...,12}{
        \node[draw,circle,inner sep=1.1pt,fill=black] at (0.5*\x,0) {};
         }
\node[right] at (6.5,0) {\small{$S_1\subseteq \ZZ$}}; 
    \coordinate (Origin)   at (0,-1);
    \coordinate (XAxisMin) at (-0.1,-1);
    \coordinate (XAxisMax) at (6.3,-1);
    \draw [thick] (XAxisMin) -- (XAxisMax);
    \foreach \x in {0,1,2,...,12}{
        \node[draw,circle,inner sep=1.4pt,fill=white] at (0.5*\x,-1) {};
         }
    \foreach \x in {0,3,6,9,12}{
        \node[draw,circle,inner sep=1.1pt,fill=black] at (0.5*\x,-1) {};
         }
\node[right] at (6.5,-1) {\small{$S_2\subseteq \ZZ$}};

    \coordinate (Origin)   at (0,-2);
    \coordinate (XAxisMin) at (-0.1,-2);
    \coordinate (XAxisMax) at (6.3,-2);
    \draw [thick] (XAxisMin) -- (XAxisMax);
       
    \foreach \x in {0,1,2,...,12}{
        \node[draw,circle,inner sep=1.4pt,fill=white] at (0.5*\x,-2) {};
         }
    \foreach \x in {0,6,9,12}{
        \node[draw,circle,inner sep=1.1pt,fill=black] at (.5*\x,-2) {};
         }

\node[right] at (6.5,-2) {\small{$S_{1}\cap S_2\subseteq \ZZ$}};

\node[draw,circle,inner sep=1pt,fill=black] at (11.5,-.3) {};  
\node[right] at (11.7,-.3) {\small {element of the monoid}}; 

\node[draw,circle,inner sep=1.4pt,fill=white] at (11.5,-1.2){};  
\node[right] at (11.7,-1.2) {\small{element of $\ZZ$ but not }}; 
\node[right] at (11.7,-1.85) {\small{element of the monoid}};  

  \end{tikzpicture}
\end{figure}
\end{center}
\end{example}

\begin{algo}[inCondIdeal] \label{algo:incondid}
\textit{Input:} A finitely generated abelian group~$K$,
generators~$s_1,\ldots,s_t\in K$
of an embedded monoid 
$S\sei\lin_{\ZZ_{\ge 0}}(s_1,\ldots,s_{t})\subseteq K$ and 
an element $w\in K$.\\
\textit{Output:} \textit{True} if $w$ is contained in $c(\tilde{S}/S)$. Otherwise, \textit{false} is returned.

\begin{itemize}
\item Compute $M$ as defined in Lemma~\ref{lem:tildeSfgmodule}.
\item Use Algorithm~\ref{algo:inmon} to test whether $S$ contains
$w+M$. Return \textit{true} if this is the case; otherwise 
return \textit{false}.
\end{itemize}

\end{algo}

\begin{proof}
Let~$w\in K$ and consider~$M$ as defined in 
Lemma~\ref{lem:tildeSfgmodule}.
According to this lemma, $M$ generates $\tilde{S}$ as an $S$-module.
This means that the conductor ideal~$c(\tilde{S}/S)$ contains $w$
if and only if $w+M$ is contained in $S$.
\end{proof}

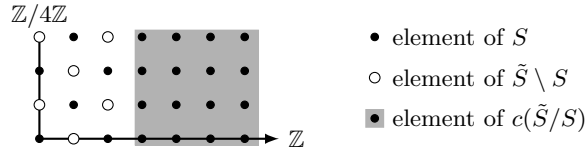
\begin{figure}[ht]
\begin{tikzpicture}[scale=0.9]
    
 \draw[fill, black!30!white] (1.4,-.1)--(1.4,1.6)--(3.2, 1.6)--(3.2,-.1);    
    
    \coordinate (Origin)   at (0,0);
    \coordinate (XAxisMin) at (0,0);
    \coordinate (XAxisMax) at (3.5,0);
    \coordinate (YAxisMin) at (0,0);
    \coordinate (YAxisMax) at (0,1.5);
    \draw [thick, -latex] (XAxisMin) -- (XAxisMax);
    \draw [thick] (YAxisMin) -- (YAxisMax);
    \foreach \x in {0,2,3,...,6}{
        \node[draw,circle,inner sep=1pt,fill=black] at (0.5*\x,0) {};
         }
    \foreach \x in {1,3,4,...,6}{
        \node[draw,circle,inner sep=1pt,fill=black] at (0.5*\x,.5) {};
         }
    \foreach \x in {0,2,3,...,6}{
        \node[draw,circle,inner sep=1pt,fill=black] at (0.5*\x,1) {};
         }
    \foreach \x in {1,3,4,...,6}{
        \node[draw,circle,inner sep=1pt,fill=black] at (0.5*\x,1.5) {};
         }         
    \node[draw,circle,inner sep=1.4pt,fill=white] at (0,0.5) {};
    \node[draw,circle,inner sep=1.4pt,fill=white] at (0,1.5) {};    
    \node[draw,circle,inner sep=1.4pt,fill=white] at (0.5,0) {};
    \node[draw,circle,inner sep=1.4pt,fill=white] at (0.5,1) {};
    \node[draw,circle,inner sep=1.4pt,fill=white] at (1,0.5) {};
    \node[draw,circle,inner sep=1.4pt,fill=white] at (1,1.5) {};

\node[right] at (3.5,0) {\small{$\ZZ$}};  
\node[above] at (0,1.5) {\small{$\ZZ/4\ZZ$}};

\node[draw,circle,inner sep=1pt,fill=black] at (4.9,1.5) {};  
\node[right] at (4.9,1.5){\small{~element of $S$}}; 

\node[draw,circle,inner sep=1.4pt,fill=white] at (4.9,.9) {}; 
\node[right] at (4.9,.9) {\small{~element of $\tilde{S}\setminus S$}}; 

\node[draw, rectangle, draw= gray!60!, fill=gray!60!] at (4.9,.3) {}; 
\node[right] at (4.9,.3) {\small{~element of $c(\tilde{S}/S)$}};  
\node[draw,circle,inner sep=1pt,fill=black] at (4.9,.3) {};        

  \end{tikzpicture}
\end{figure}

\begin{example}\label{ex:2}
Consider the abelian group~$K\sei\ZZ\oplus\ZZ/4\ZZ$ as well as its 
elements~$s_1\sei (0,\bar{2})$, $s_2\sei (1,\bar{1})$, $s_3\sei (3,\bar{2})$
and the monoid $S\sei\lin_{\ZZ_{\ge 0}}(s_1,s_2,s_3)$ as in Example~\ref{ex:1}.
The monoid and its conductor ideal are illustrated in the above picture.
We apply algorithm~\ref{algo:incondid} to $w\sei (3,\bar{1})$ and test whether
$w$ is contained in $c(\tilde{S}/S)$.
\begin{itemize}
\item The maps~$Q$ and~$Q^0$ are as in Example~\ref{ex:1}.
\item The algorithm computes $M$ as defined in Lemma~\ref{lem:tildeSfgmodule}. 
We obtain 
$$
M \ = \ \{(0,a), \, (1,a); \; a\in \ZZ/4\ZZ\} \ \subseteq \ K \,.
$$
\item 
In the next step the algorithm uses Algorithm~\ref{algo:inmon} to test 
whether $S$ contains $w+M=\{(3,a), \, (4,a); \; a\in \ZZ/4\ZZ\}$. 
\item Similarily as in Example~\ref{ex:1}, for $x\in w+M$ with $x^0=3$,
Algorithm~\ref{algo:inmon} computes
$
B_3 \sei \{ (\alpha, 3,0),\, (\alpha, 0,1); \; 
0\le \alpha\le 4, \, \alpha \in \ZZ\}\, 
$
and we obtain
$
B_4 \sei
\{ (\alpha, 4,0),\, (\alpha, 1,1); \; 
0\le \alpha\le 4, \, \alpha \in \ZZ\}\, 
$
for all $x\in w+M$ with $x^0=4$.
Since for all $x\in w+M$ with $x^0=i$, $i=3,4$, there is some $y\in B_i$
with $Q(y_i)=x_i$, the algorithm returns \textit{true}.
\end{itemize}
\end{example}

\begin{algo}[pointCondIdeal] \label{algo:pointcondid}
\textit{Input:} A finitely generated abelian group~$K$,
an element $w\in K$ and
generators~$s_1,\ldots,s_t\in K$
of a spanning embedded monoid 
$S\sei\lin_{\ZZ_{\ge 0}}(s_1,\ldots,s_{t})\subseteq K$.\\
\textit{Output:} A point of the conductor ideal $c(\tilde{S}/S)$.

\begin{itemize}
\item Compute $w\in K$ that defines a point in the relative interior 
of $\cone(S)$.
\item Use Algorithm~\ref{algo:incondid} to compute 
the smallest integer $r \in \ZZ_{\ge 1}$ such that $rw$ is contained in
$c(\tilde{S}/S)$. Return $rw$.
\end{itemize}
\end{algo}

\begin{proof}
This Algorithm terminates since $S\subseteq K$ is spanning.
\end{proof}

\begin{example}
Consider the abelian group~$K\sei\ZZ\oplus\ZZ/4\ZZ$ as well as its 
elements~$s_1\sei (0,\bar{2})$, $s_2\sei (1,\bar{1})$, $s_3\sei (3,\bar{2})$
and the monoid $S\sei\lin_{\ZZ_{\ge 0}}(s_1,s_2,s_3)$ as in Examples~\ref{ex:1} and~\ref{ex:2}.
We apply algorithm~\ref{algo:pointcondid} to compute an element of~$c(\tilde{S}/S)$.
\begin{itemize}
\item At first the algorithm computes the element $(1,\bar{0})\in K$ defining 
an element in the relative interior of $\cone(S)$.
\item 
For $j=1,2$, Algorithm~\ref{algo:incondid} returns that
$j\, (1,\bar{0})$ is not contained in~$c(\tilde{S}/S)$.
\item In the next step, Algorithm~\ref{algo:incondid} shows that
$(3,\bar{0})$ is an element of~$c(\tilde{S}/S)$.
\end{itemize}
\end{example}

\section{The base point free monoid of 
Mori dream spaces}

We turn to Mori dream spaces and recall the necessary 
background from~\cite{ADHL}. In addition, we 
present a description of the monoid
of base point free Cartier divisor classes of a Mori dream space
in terms of combinatorial data,
which will be crucial in Algorithm~\ref{algo:fujitabpf}.
We also present algorithms for computing generators of
the base point free monoid and for testing whether
a Weil divisor class is base point free or~not.

\begin{definition}
Let $D$ be a Weil divisor on an irreducible,
normal variety $X$ and consider 
a non-zero section 
$f \in \Gamma(X, \mathcal{O}_{X}(D))$.
We call the effective divisor
$$
{\rm div}_D(f)
\ \sei \
{\rm div}(f) + D 
\ \in \
\WDiv(X)
$$
the \emph{$D$-divisor} of $f$.
The \emph{base locus} and the 
\emph{stable base locus} 
of the class 
$w \sei [D]  \in \Cl(X)$ 
 are defined as
$$
\Bs(w) 
\  \sei \
\bigcap_{f \in \Gamma(X, \mathcal{O}_{X}(D))} \Supp({\rm div}_D(f)),
\ \ \ \ \ \
\B(w) 
\  \sei \
\bigcap_{n \in \ZZ_{\ge 0}} \Bs|nD|\,.
$$
An element $x \in  \Bs(w)$ is 
called a $\emph{base point}$
of $w$. We call $D \in \WDiv(X)$ or its 
class~$w \in \Cl(X)$ \emph{base point free}
if the base locus $\Bs(w)$ is empty
and \emph{semiample} if its stable 
base locus is empty.
The embedded monoid $\BPF(X)\subseteq\Pic(X)$
of base point free Cartier divisor classes is called 
\emph{base point free monoid of $X$}.
By $\SAmple(X)\subseteq\Cl(X)_{\QQ}$ 
and $\Ample(X)\subseteq\Cl(X)_{\QQ}$, we denote 
the cones of semiample and ample
Weil divisor classes, respectively. 
\end{definition}

Recall that a \emph{Mori dream space} is an irreducible normal 
projective variety~$X$ over an algebraic closed field of characteristic
zero with finitely generated divisor class 
group and finitely generated Cox ring
$$
{\rm Cox}(X)
\ \sei \
\bigoplus_{[D] \in \Cl(X)} 
\Gamma( X, \mathcal{O}_X(D))\, ,
$$
where in case of torsion in the divisor class group 
some care is required in this definition, see \cite[Section 1.4]{ADHL}.
Recall that a non-zero non-unit $f\in{\rm Cox}(X)$ is called~${\rm Cl}(X)$-prime 
if it is homogeneous and if $f|gh$ with homogeneous elements~$g,h\in{\rm Cox}(X)$ 
implies~$f|h$ or $f|g$.
Let $\mathfrak{F}\sei(f_1,\ldots,f_r)$ be a system of pairwise
non-associated ${\rm Cl}(X)$-prime generators of ${\rm Cox}(X)$.
Consider the homomorphism of abelian groups
$Q\colon \ZZ^r\to{\rm Cl}(X), \; e_i \mapsto\deg(f_i)$ as well as 
the total coordinate space $\overline{X}\sei\Spec({\rm Cox}(X))$.
A face $\gamma_0 \preceq \gamma$ of the positive
orthant~$\gamma \subseteq \QQ^r$ is called an~\emph{$\mathfrak{F}$-face},
if there is a point $x\in\overline{X}$ with 
$x_i \neq 0 \Leftrightarrow e_i \in\gamma_0$ for all 
$1\le i\le r$.
The collection of \emph{relevant faces} and the 
\emph{covering collection} are
$$
\rlv(X) 
\ \sei \ 
\{
\gamma_0\preceq \gamma; \ \gamma_0 \; \mathfrak{F}
\text{-face with } \Ample(X)^\circ\subseteq Q(\gamma_0)^\circ
\}\, ,
$$
$$
\cov(X) 
\ \sei \ 
\{
\gamma_0\in\rlv(X); \ \gamma_0 \text{ minimal with respect to inclusion}
\}\, .
$$
Note that an ample Weil divisor class~$u\in{\rm Cl}(X)$ 
together with $\mathfrak{F}$ and relations 
of~${\rm Cox}(X)$
fixes a Mori dream space up to isomorphism: 
one can reconstruct~$X$ as GIT-quotient 
$p_X\colon \overline{X}^{\rm ss}(u)\to X$ of the set of 
$H$-semistable points $\overline{X}^{\rm ss}(u) \subseteq \overline{X}$
regarding the action of $H\sei \Spec\KK(\Cl(X))$
on~$\overline{X}$, see for instance~\cite[Section~3.1]{ADHL}.
 
To any relevant face $\gamma_0\preceq \gamma$
we associate as in~\cite[Construction 3.3.1.1]{ADHL} the set $X(\gamma_0) \sei p_X(\overline{X}(\gamma_0))$, where we have
$$
\overline{X}(\gamma_0)
\ \sei \ 
\{ x \in \overline{X}; \ \ f_i(x)\neq 0 \Leftrightarrow e_i \in\gamma_0 \text{ for } 1\le i\le r\}
\ \subseteq \
\overline{X}^{\rm ss}(u)\, .
$$
According to~\cite[Corollary 3.3.1.6.]{ADHL} 
and to~\cite[Proposition 3.3.2.8]{ADHL},
the Picard group~$\Pic(X)$ of~$X$ and the base locus
of an element $w\in\Cl(X)$ are given by
$$ 
\Pic(X) 
\ = \ 
\bigcap_{\gamma_0 \in \cov(X)} 
Q(\lin(\gamma_0) \cap E)
\quad \text{ and } \quad
\Bs(w) 
\  = \
\bigcup_{\gamma_0 \in{\rm rlv}(X) 
\atop w\notin Q(\gamma_0 \cap \ZZ^r)} 
X(\gamma_0)\,.
$$

For projective varieties, any Cartier divisor is the difference
of two very ample divisors~\cite[1.20]{debarre}.
Thus, the base point free monoid of projective varieties is
a spanning embedded monoid. By Proposition~\ref{prop:condid}, this 
means in particular that 
its conductor ideal is non-empty.
For Mori dream spaces, 
we retrieve the same result in the
following Corollary. In addition, we give
a description of~$\BPF(X)$ in terms of the covering collection
and the homomorphism 
$Q\colon \ZZ^r\to{\rm Cl}(X), \; e_i \mapsto\deg(f_i)$.

\begin{corollary}\label{cor:bpfspanning}
In the above notation, the  base point free monoid $\BPF(X) \subseteq \Pic(X)$
of a Mori dream space~$X$
is a spanning embedded monoid
given by 
$$ 
\BPF(X) 
\ = \ 
\bigcap_{\gamma_0 \in \cov(X)} Q(\gamma_0 \cap \ZZ^r)\,.
$$
\end{corollary}

\begin{proof}
The representation of $\BPF(X)$ as an intersection of
monoids $Q(\gamma_0\cap E)$ is an immediate consequence
of the above description of base loci.
In addition, for each $\gamma_0 \in \cov(X)$, 
the embedded 
monoid~$Q(\gamma_0 \cap E) \subseteq Q(\lin(\gamma_0) \cap E)$ 
is spanning.
Using the above description of $\Pic(X)$ 
together with Proposition~\ref{prop:intersection}, 
we obtain that $\BPF(Z) \subseteq \Pic(Z)$ 
is a spanning embedded monoid.
\end{proof}

Note that the following algorithms build on the maple-based
software package {\tt{MDSpackage}}~\cite{mdspackage}.
A Mori dream space~$X$ is entered and stored in terms of 
an ample class~$u$ together with pairwise 
non-associated~$\Cl(X)$-prime generators and relations 
of~${\rm Cox}(X)$.
As explained above, this data fixes a Mori dream space up to isomorphism.

\begin{algo}[generatorsBPF] \label{algo:genBPF} 
\textit{Input:} A Mori dream space.\\
\textit{Output:} A set of generators for the embedded monoid
$\BPF(X)\subseteq\Pic(X)$.

\begin{itemize}
\item 
Use {\tt{MDSpackage}} to compute the covering collection of~$X$.
\item 
Use Algorithm~\ref{algo:genintmon} to compute generators 
of the intersection
$$\bigcap_{\gamma_0\in\cov(X)} Q(\gamma_0 \cap \ZZ^r)\,.$$
\end{itemize}
\end{algo}

\begin{algo}[isBasePointFree] \label{algo:isbasepointfree}
\textit{Input:} A Mori dream space~$X$ and a Weil divisor class $w\in\Cl(X)$.\\
\textit{Output:} \textit{True} if $w$ is base point free. 
Otherwise, \textit{false} is returned.

\begin{itemize}
\item 
Use Algorithm~\ref{algo:genBPF} to compute generators
of $\BPF(X)\subseteq\Pic(X)$.
\item 
Apply Algorithm~\ref{algo:inmon} to $w$ and $\BPF(X)$.
\end{itemize}
\end{algo}

Using the implementation given in~\cite{MonoidPackage}, we study the 
question of the existence of semiample Cartier divisor classes 
that are not base point free. It is well-known that for Cartier divisors 
on complete 
toric varieties, semiampleness implies base point freeness,
see for instance~\cite[Theorem 6.3.12.]{CLS}.
For smooth rational projective varieties with a torus action 
of complexity one and Picard number two, the same statement
follows immediately from the classification done in~\cite{fahani}.
Note that the discrepancy between semiampleness and base point 
freeness of divisors on varieties with a torus action of complexity
one is already fairly well understood in the language of polyhedral divisors: 
A criterion for semiampleness is given in~\cite[Theorem 3.27]{PeSu}
and a criterion for base point freeness was proved in~\cite[Theorem 3.2]{IlVo}.

\begin{example}\label{ex:smoothsemiamplenotbpf}
We give an example of a smooth Mori dream $\KK^*$-surface that admits semiample Cartier 
divisor classes with base points.
\begingroup
\footnotesize
\setlength{\arraycolsep}{4pt}
\begin{enumerate}[leftmargin=3em]
\item[\tt >] \tt{Q $\sei$ matrix([[1,-1,-1,0,0,0,0,0,0,0,0,0,0,0,0],[0,1,-1,1,0,0,0,0,0,0,0,0,0,0,0],
[0,1,0,-1,1,0,0,0,0,0,0,0,0,0,0],[0,1,0,0,-1,1,0,0,0,0,0,0,0,0,0],[0,0,0,0,0,0,  \\ -1,1,1,0,0,0,0,0,0],[0,-1,0,0,0,1,0,-1,1,0,0,0,0,0,0],[0,0,0,1,0,0,1,0,1,1,0,0, \\ 0,0,0],[0,1,0,0,0,0,0,0,1,0,1,0,0,0,0],[1,0,0,-1,0,0,1,0,0,0,0,1,0,0,0],[0,1,0,\\ 0,0,0,0,1,0,0,0,0,1,0,0],[0,1,0,0,0,-1,0,0,0,0,0,0,0,1,0],[0,-1,0,0,0,1,0,0,0,0, 0,0,0,0,1]]);}
\out{Q\ \sei \
\left[ 
\begin {array}{ccccccccccccccc}
 1&-1&-1&0&0&0&0&0&0&0&0&0&0&0&0
\\ \noalign{}0&1&-1&1&0&0&0&0&0&0&0&0&0&0&0
\\ \noalign{}0&1&0&-1&1&0&0&0&0&0&0&0&0&0&0
\\ \noalign{}0&1&0&0&-1&1&0&0&0&0&0&0&0&0&0
\\ \noalign{}0&0&0&0&0&0&-1&1&1&0&0&0&0&0&0
\\ \noalign{}0&-1&0&0&0&1&0&-1&1&0&0&0&0&0&0
\\ \noalign{}0&0&0&1&0&0&1&0&1&1&0&0&0&0&0
\\ \noalign{}0&1&0&0&0&0&0&0&1&0&1&0&0&0&0
\\ \noalign{}1&0&0&-1&0&0&1&0&0&0&0&1&0&0&0
\\ \noalign{}0&1&0&0&0&0&0&1&0&0&0&0&1&0&0
\\ \noalign{}0&1&0&0&0&-1&0&0&0&0&0&0&0&1&0
\\ \noalign{}0&-1&0&0&0&1&0&0&0&0&0&0&0&0&1
\end {array}
 \right]
}

\item[\tt >] \tt{RL := [T[1]\textasciicircum 5*T[2]*T[3]\textasciicircum 4*T[4]\textasciicircum 3*T[5]\textasciicircum 2*T[6]
+T[7]\textasciicircum 2*T[8]*T[9]\\
+T[10]\textasciicircum 3*T[11]*T[12]\textasciicircum 2*T[13]];}
\out{RL := \left[
T_1^5T_2T_3^4T_4^3T_5^2T_6
+T_7^2T_8T_9
+T_{10}^3T_{11}T_{12}^2T_{13}\right]}
\item[\tt >] \tt{R $\sei$ createGR(RL, 
vars(15),
[Q]);}
\out{R\sei GR(15,1,[12,[]])}
\item[\tt >] \tt{X $\sei$ createMDS(R,relint(MDSmov(R)));}
\out{X \sei MDS(15,1,2,[12,[]])}
\item[\tt >] \tt{MDSissmooth(X);}
\out{true}
\item[\tt >] \tt{w $\sei$ [-1,1,1,1,3,2,3,4,0,3,1,5];}
\out{w := [-1, 1, 1, 1, 3, 2, 3, 4, 0, 3, 1, 5]}
\item[\tt >] \tt{contains(MDSsample(X),w);}
\out{true}
\item[\tt >] \tt{isBasePointFree(X,w);}
\out{false}
\end{enumerate}
\endgroup
\noindent
The computation shows that 
$w=[-1,1,1,1,3,2,3,4,0,3,1,5]$
is a semiample but not base point free Cartier divisor class.

For a geometric interpretation note that~$X$
is obtained by blowing up~$\PP_1\times\PP_1$ ten times in the
following way:
One considers the~$\KK^*$-action on~$\PP_1\times\PP_1$
given by
$$
t \cdot ([y_0,y_1],[z_0,z_1]) 
\ \sei \ ([y_0,y_1],[z_0,tz_1]) \, .
$$
The fixed points lie on the two curves
$C_1\sei \PP_1\times\{[0,1]\}$
and~$C_1\sei \PP_1\times\{[1,0]\}$.
In order two obtain~$X$, one blows up
the three fixed points
$c_{11} \sei \{[0,1],[0,1]\} \in C_1$, 
$c_{12} \sei \{[1,0],[0,1]\} \in C_1$ and 
$c_{21} \sei \{[1,-1],[1,0]\}  \in C_2$. 
The resulting hyperbolic fixed points
are again blown up:
for~$c_{11}$, one repeats this 
four times, 
for~$c_{12}$, one repeats this 
two times
and for~$c_{21}$ just once.
The resulting variety then is isomorphic to~$X$.
\end{example}

\section{Fujita base point free test}

In the end of the eighties, Takao Fujita conjectured the following:

\begin{conjecture}[Fujita's base point free conjecture~\cite{fuconj}] \label{conj:bpf}
Let $X$ be an n-dimensional smooth projective variety 
with canonical class $\mathcal{K}_X$ and let 
$\mathcal{L}$ be an ample Cartier divisor 
class. Then the following holds:
$$
\mathcal{K}_X + m  \mathcal{L} \text{ is base point free 
for all } m \geq n+1.
$$
\end{conjecture}
In order to test whether a~$\QQ$-factorial Mori dream 
space~$X$ with known canonical class fulfills 
Fujita's base point free conjecture,
we need to test whether $\mathcal{K}_X + m  \mathcal{L}$ is an 
element of~$\BPF(X)$ for all $m \geq \dim(X)+1$ and for all 
ample Cartier divisor classes $\mathcal{L}$.
Since we can only carry out finitely many tests, we encouter 
two problems: firstly, we need to bound $m$ and secondly, we 
need to find a finite validation set 
of Cartier divisor classes $\mathcal{L}$.
In this section, we introduce our solution to these
problems and also present some examples of
applying our test algorithm.\\

\begin{remark}~\label{rem:knownKX}
Algorithm~\ref{algo:fujitabpf} applies to Mori dream spaces with 
known canonical class.
For instance, if~${\rm Cox}(X)$ is a complete intersection, 
there is a concrete formula for the canonical class 
in terms of generators and relations of~${\rm Cox}(X)$~\cite[Proposition 3.3.3.2]{ADHL}.
Note that all irreducible normal rational projective 
varieties with a torus action of complexity one 
have a complete intersection Cox ring~\cite[Proposition 1.2]{HaHeSu}.
In addition, there are formulas for the canonical class
of spherical varieties, see~\cite{Brion, Luna}.
\end{remark}

\begin{construction}\label{facetpar}
Let $K^0$ be a lattice. Consider an $s$-dimensional cone 
$\sigma \subseteq K^0_{\QQ}$ with some facet 
$F\preceq\sigma$.
Let $\varphi\colon K^0\to \ZZ^n$ 
be an isomorphism of $\ZZ$-modules
such that 
$\varphi(\sigma)\subseteq\cone(e_1,\ldots,e_{s})$ 
and
$\varphi(F)\subseteq\cone(e_1,\ldots,e_{s-1})$
holds, where $e_1,\ldots,e_n$ denote the canonical base vectors
of the rational vector space $\QQ^n$.
For any~$k\in\ZZ$ 
we call~$\tau\sei\varphi^{-1}(
\tilde{\tau})$ the \emph{k-th facet parallel} 
of $F$, where we set
$$
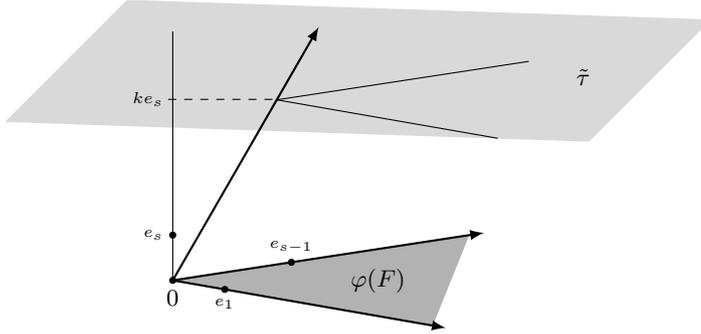
\tilde{\tau}
\; \sei \;
\big(
\lin_{\QQ}(\varphi(F))
\ +\
k e_s
\big)
\, .
$$ 
\end{construction}

\begin{figure}[ht]
\begin{tikzpicture}[scale=0.6]

     
    \path[fill=gray!60!] (0,0)--(5.7,-1)--(6.5,1)--(0,0); 
    
    \path[fill=gray!30!] ($(-2,5.2)+(1,1)$) -- 
    ($(-2,5.2) + (-1,0)
    -0.57*(4, 7) - 0.85*(6.5,1)
    +0.57*(4, 7) + 0.85*(5.7,-1)
    $) 
    -- ($0.57*(4, 7) + 0.85*(5.7,-1) +(2,-.075)$)
    --($0.57*(4, 7) + 0.85*(6.5,1)+(3,-.075)+(1,1)$);
    
    \draw[thick, -latex]  (0,0) -- ($1.05*(5.7,-1)$); 
    \draw[thick, -latex]  (0,0) -- ($1.05*(6.5,1)$);
    \node[] at (4.5,0) {\small{$\varphi(F)$}};     
    \node[draw,circle,inner sep=.8pt,fill=black] at (0,0) {}; 
    \node[below] at (0,0) {\small{$0$}};       
    \coordinate (esm1) at (.4*6.5, .4*1);
    \node[above] at (esm1) {\tiny{$e_{s-1}$}};   
    \node[draw,circle,inner sep=.8pt,fill=black] at (esm1) {};  
    \coordinate (e1) at (.2*5.7,.2*-1);
    \node[below] at (e1) {\tiny{$e_{1}$}}; 
    \node[draw,circle,inner sep=.8pt,fill=black] at (e1) {};  
    \draw[thick, -latex]  (0,0) -- ($0.8*(4,7)$); 

    \draw[]  ($0.57*(4,7)$) -- ($0.57*(4, 7) + 0.85*(5.7,-1)$);
    \draw[]  ($0.57*(4,7)$) -- ($0.57*(4, 7) + 0.85*(6.5,1)$); 
    \node[] at (9,4.5) {\small{$\tilde{\tau}$}};
  
    \draw[thin]  (0,0) -- (0,5.5);
    \node[left] at (0,1) {\tiny{$e_{s}$}};
    \node[draw,circle,inner sep=.8pt,fill=black] at (0,1) {}; 
      
    \draw [thin, dashed]  (-.1, 4) -- (2.3, 4);
    \node[left] at (0,4) {\tiny{$k e_{s}$}};
  \end{tikzpicture}
\end{figure}

\begin{setting}\label{set:algo}
Let $X$ be a~$\QQ$-factorial
Gorenstein Mori dream space and 
consider the base point free monoid
$S \sei \BPF(X) \subseteq K \sei \Pic(X)$.
We denote by
$F_1,\ldots,F_r$ the facets of 
$\sigma \sei  \cone(w^0 \otimes 1; \ w\in S) \subseteq K^0_{\QQ}$.
Let $1\le i\le r$
and let~$m_1,\ldots,m_{n_i} \in S$ be those elements such that 
$m_j^0$ is minimal with the property that $m_j^0 \otimes 1$ 
is contained in a ray of $F_i$.
Consider the polytope
$$
G_i
\ \sei \ 
\{ \sum_{j=1}^{n_i} a_j \, (m_j^0 \otimes 1); \ \;
 a_j \in \QQ, \ 0\le a_j\le 1 \}
 \ \subseteq \ 
 F_i
$$
as indicated in the figure below
and let $\rho_1,\ldots\rho_{t_i}$ be the 
rays of~$\sigma$ that are not contained in $F_i$. 
We denote by $\tau_i^k$ the k-th facet parallel of $F_i$.
For each facet parallel $\tau_i^k$ with~$k\in \ZZ_{\ge 0}$, 
we denote by $p_j^k\in K_{\QQ}, \, 1\le j\le t_i,$ the point that is 
the intersection of $\rho_j$ and~$\tau_i^k$.
With the canonical embedding 
$\iota_0\colon K^0 \to K^0_{\QQ}, \, w\mapsto w\otimes 1$,
we define
\begin{eqnarray*}
P_i^k 
& \sei &
\left({\rm conv}(p_1^k, \ldots, p_{t_i}^k) + G_i \right)
\; \cap \;  
\sigma^\circ
 \ \subseteq \ 
 \tau_i^k \
\ \text{ and } \\
Gp_i^k 
& \sei &
\iota_0^{-1}
\left( 
P_i^k
\right)
\ \times \
K^{\rm tor}
\ \subseteq \ 
K
\end{eqnarray*}
for all $k\in \ZZ_{\ge 0}$,
where $\sigma^\circ$ denotes the relative interior of $\sigma$.
Consider the canonical class $\mathcal{K}_X\in K$ of~$X$.
Since~$S\subseteq K$ is spanning, there is
an element $C \in c(\tilde{S}/S)$.
For $1\le i\le r$ let $\alpha_i\in\ZZ$ such 
that~$(-\mathcal{K}_X^0+C^0) \otimes 1 \in \tau_i^{\alpha_i}$ holds and set
$\nu \sei \max(\alpha_i; \ 1\le i\le r)$.
Note that~$\alpha_i$ may be negative.
\end{setting}

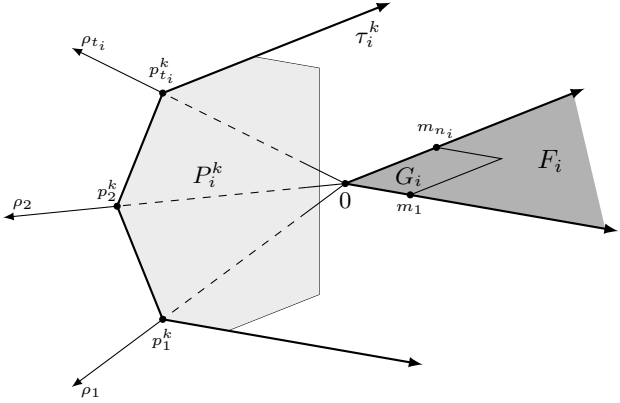
\begin{figure}[ht]
\begin{tikzpicture}[scale=0.6]
    \path[fill=gray!60!] (0,0)--(5.7,-1)--(5,2)--(0,0); 
    \draw[thick, -latex]  (0,0) -- ($1.05*(5.7,-1)$); 
    \draw[thick, -latex]  (0,0) -- ($1.05*(5,2)$);
    \node[] at (4.5,.5) {$F_i$};     
    \node[draw,circle,inner sep=.8pt,fill=black] at (0,0) {}; 
    \node[below] at (0,0) {\small{$0$}};       
    \coordinate (mni) at (.4*5, .4*2);
    \node[above] at (mni) {\tiny{$m_{n_i}$}};   
    \node[draw,circle,inner sep=.8pt,fill=black] at (mni) {};  
    \coordinate (m1) at (.25*5.7,.25*-1);
    \node[below] at (m1) {\tiny{$m_{1}$}}; 
    \node[draw,circle,inner sep=.8pt,fill=black] at (m1) {};  
    \coordinate (m1pmni) at ($(m1)+(mni)$);
    \draw[thin]  (mni) -- (m1pmni) -- (m1); 
    \node[] at (1.4,.15) {\small{$G_i$}}; 
    \draw[thin]  ($ (-4, 2)+(mni)$) -- ($ (-4, 2)+(m1pmni)$) --
    ($ (m1pmni)+(-4, -3)$) -- ($(-4, -3)+(m1)$); 
    \path[fill=gray!15!] 
           (-4, -3) -- (-5, -.5) -- (-4, 2) -- ($ (-4, 2)+(mni)$) 
           -- ($ (-4, 2)+(m1pmni)$) -- ($ (m1pmni)+(-4, -3)$) 
           -- ($(-4, -3)+(m1)$) -- (-4, -3);
    \node[] at (-3,0.2) {\small{$P_i^k$}}; 
    \draw [thin] (0,0) --  (-.21*4, .21*2);
    \draw [thin, dashed]  (-.21*4, .21*2) -- (-4, 2);
    \draw [thin, -latex] (-4, 2) --  (-1.5*4, 1.5*2);
    \node[above,right] at (-1.5*4, 1.5*2+.1) {\tiny{$\rho_{t_i}$}};
    \draw [thin] (0,0) --  (-.17*5, -.17*.5);
    \draw [thin, dashed]  (-.17*5, -.17*.5) -- (-5, -.5);
    \draw [thin, -latex] (-5, -.5) --  (-1.5*5, -1.5*.5);
    \node[above, right] at (-1.5*5, -1.5*.5+.3) {\tiny{$\rho_{2}$}};
    \draw [thin] (0,0) --  (-.21*4, -.21*3);
    \draw [thin, dashed]  (-.21*4, -.21*3) -- (-4, -3);
    \draw [thin, -latex] (-4, -3) --  (-1.5*4, -1.5*3);
    \node[above, right] at (-1.5*4, -1.5*3-.1) {\tiny{$\rho_{1}$}};
    \draw [thick, -latex]  (-4, 2) -- (-4+5, 2+2);
    \draw [thick] (-4, 2)-- (-5, -.5);
    \draw [thick] (-5, -.5) -- (-4, -3);
    \draw [thick, -latex] (-4, -3) -- (-4+5.7, -3-1);
    \node[above] at (-4,2) {\tiny{$p_{t_i}^k$}};
    \node[draw,circle,inner sep=.8pt,fill=black] at (-4,2) {};   
    \node[above] at (-5.2,-.6) {\tiny{$p_{2}^k$}}; 
    \node[draw,circle,inner sep=.8pt,fill=black] at (-5,-.5) {};  
    \node[below] at (-4,-3) {\tiny{$p_{1}^k$}}; 
    \node[draw,circle,inner sep=.8pt,fill=black] at (-4,-3) {};   
    \node[] at (0.5,3.3) {\small{$\tau_i^k$}};  
  \end{tikzpicture}
\end{figure}

The above mentioned problems, namely bounding $m$ and finding
a finite validation set of Cartier divisor classes, are tackled
by computing a point of the conductor ideal of 
$\BPF(X)$ and by only considering the Cartier divisor
classes defining a point in the polytopes $P_i^k$ of the first few
facet parallels $\tau_i^k$, $k\ge 0$, of each facet $F_i \preceq \sigma$.

\begin{algo}[fujitaBpf] \label{algo:fujitabpf}
\textit{Input:} A~$\QQ$-factorial
Mori dream space~$X$ and its canonical 
class~$\mathcal{K}_X$.\\
\textit{Output:} \textit{True} if $X$ fulfills Fujita's base point free conjecture,
i.e.~if $\mathcal{K}_X + m  \mathcal{L}$ is base point free 
for all $m \geq \dim(X)+1$ and all ample Cartier divisor classes $\mathcal{L}$.
Otherwise, \textit{false} is returned.
\begin{itemize}
\item 
If $X$ is not Gorenstein return \textit{false}.
\item 
Use Algorithm~\ref{algo:genintmon} to compute generators 
of $S\sei\BPF(X)$.
\item
Use Algorithm~\ref{algo:pointcondid} to compute a point
$C\in c(\tilde{S}/S)$. 
\item 
Compute the facets $F_1,\ldots,F_r$ of \cone(S) and
$\alpha_1,\ldots,\alpha_r$ as well as $\nu$ as defined in Setting~\ref{set:algo}.
\item
For each $1\le i\le r$ do
\begin{itemize}
\item for each $\dim(X)+1\le m \le\nu-1$ do
\begin{itemize}
\item for each 
$1\le k \le \lfloor\frac{\alpha_i -1}{m} \rfloor$, where ~$\lfloor\cdot\rfloor$ denotes the floor function, use 
Algorithm~\ref{algo:inmon} to
test whether $\mathcal{K}_X+m \, Gp_i^k \subseteq S$ holds.
\end{itemize}
\end{itemize}
\item Return \textit{false} if there is $1\le i\le r$, 
$\dim(X)+1\le m\le \nu-1$, 
$1\le k \le  \lfloor \frac{\alpha_i -1}{m} \rfloor$, 
and $\mathcal{L} \in Gp_i^k$ such that $\mathcal{K}_X+m  \mathcal{L}$ 
is not contained in $S$. Otherwise, return \textit{true}.
\end{itemize}

\end{algo}

Before presenting a proof of Algorithm~\ref{algo:fujitabpf},
we first give two examples of applying it to 
Mori dream spaces.

\begin{example}\label{ex:fujitabpf1}
Here we give an example of a six-dimensional 
smooth Mori dream space that does fulfill
Fujita's base point free conjecture.\\
\begingroup
\footnotesize
\setlength{\arraycolsep}{4pt}
\begin{enumerate}[leftmargin=3em]
\item[\tt >] \tt{Q $\sei$ matrix([[1,1,2,0,1,1,1,-1,0,0],[0,0,-1,1,0,-1,-1,1,0,0],[0,36,36,0,18,49,49, -48,1,1]]);}
\out{
Q =
\begin{array}{c}
\left[\!\!
\begin{array}{cccccccccc}
1 & 1 & 2 & 0 & 1 & 1 & 1 & -1 & 0 & 0
\\ 
0 & 0 & -1 & 1 & 0 & -1 & -1 & 1 & 0 & 0
\\
0 & 36 & 36 & 0 & 18 & 49 & 49 & -48 & 1 & 1
\end{array}
\!\!\right]
\end{array}
}
\item[\tt >] \tt{RL := [T[1]*T[2]+T[3]*T[4]+T[5]\textasciicircum2];}
\out{RL := \left[T_1T_2+T_3T_4+T_5^2\right]}
\item[\tt >] \tt{R $\sei$ createGR(RL,vars(10),[Q]);}
\out{R\sei GR(10, 1, [3, []])}
\item[\tt >] \tt{X $\sei$ createMDS(R,[1,1,50]);}
\out{X \sei MDS(10,1,6,[3,[]])}
\item[\tt >] \tt{MDSissmooth(X);}
\out{true}
\end{enumerate}
\endgroup
\noindent
Since $R={\rm Cox}(X)$ is a complete intersection, we may use the formula
presented in~\cite{ADHL} to compute the canonical class of~$X$:
we obtain $\mathcal{K}_X = [-4, 1, -106] \in \ZZ^3$.
\vspace{0.1cm}
\begingroup
\footnotesize
\setlength{\arraycolsep}{4pt}
\begin{enumerate}[leftmargin=3em]
\item[\tt >] \tt{fujitaBPF(X,[-4,1,-106]);}
\out{true}
\end{enumerate}
\endgroup
\noindent 
To obtain this result the algorithm performs the following steps:
\begin{itemize}
\item 
First Algorithm~\ref{algo:genintmon} is used to compute the three generators 
of $[0, 0, 1], [0, 1, 0]$ and~$[1, 0, 49]$ of $\BPF(X)\subseteq\ZZ^3$.
\item
Then Algorithm~\ref{algo:pointcondid} computes the point
$C \sei[0,0,0]\in c(\tilde{S}/S)$.  
\item 
The faces of $\cone(S)$ are given by
$F_1\sei\cone([0,1,0],[1,0,49])$, 
$F_2\sei\cone([0,0,1],[1,0,49])$, 
$F_3\sei\cone([0,1,0],[0,0,1])$.
The algorithms then computes $\alpha_1,\alpha_2,\alpha_3$
such that $-\mathcal{K}_X+C=[4,-1,106]$ defines a point
in~$\tau_i^{\alpha_i}$. 
We obtain $\alpha_1=-90$, $\alpha_2=-1$ and $\alpha_3=4$
as well as $\nu=4$. Note that $\alpha_3=4$ is just
the first coordinate of $-\mathcal{K}_X+C$.
\item
Since $\dim(X)+1=7> 4=\nu-1$ holds, the algorithm returns
 \textit{true}.
\end{itemize}
For a geometric description of~$X$, note that in the 
language of~\cite{Ca},
$X$ admits three elementary contractions two of which are birational small.
The other one is a birational divisorial contraction~$X\to Y$
contracting the divisor corresponding to the variable~$T_8$ of~${\rm Cox}(X)$. 
The variety~$Y$ is a smooth intrinsic quadric with generator degrees,
relation and semiample cone given by
\setlength{\arraycolsep}{4pt}
$$
Q \ = \
\left[
\begin{array}{cc|cc|c||cccc}
60& 0 & 48&12 & 30& 1 & 1 & 1 & 1\\
1 & 1 & 1 & 1 & 1 & 0 & 0 & 0 & 0
\end{array}
\right], 
\qquad 
g \ = \ T_1T_2+T_3T_4+T_5^2
$$
and~$\SAmple(X)=\cone((1,0),(60,1))$.
The center of~$\varphi$ is the intersection of~$Y$ and the toric
prime divisors corresponding to the variables~$T_8,T_9\in{\rm Cox}(Y)$.
Note that~$Y$ allows a closed embedding into the 
the projectivized split vector bundle
$$
\PP\big(
\mathcal{O}_{\PP_3} 
\oplus
\mathcal{O}_{\PP_3}(12)  
\oplus
\mathcal{O}_{\PP_3}(30)
\oplus
\mathcal{O}_{\PP_3}(48) 
\oplus
\mathcal{O}_{\PP_3}(60) 
\big)\, .
$$
\end{example}

\begin{example}\label{ex:fujitabpf2}
Here we give an example of a locally factorial 
variety with a torus action of complexity one that does not fulfill
Fujita's base point free conjecture.
Note that this represents a difference to the toric case,
where Fujino~\cite{fujino} presented a proof of Fujita's base
point free conjecture for toric varieties with arbitrary singularities.\\
\begingroup
\footnotesize
\setlength{\arraycolsep}{4pt}
\begin{enumerate}[leftmargin=3em]
\item[\tt >] \tt{Q $\sei$ matrix([[0,0,1,0,0,1,1,0,1],[1,1,0,1,1,0,1,1,2]]);}
\out{
Q =
\begin{array}{c}
\left[\!\!
\begin{array}{ccccccccc}
0 & 0 & 1 & 0 & 0 & 1 & 1 & 0 & 1
\\ 
1 & 1 & 0 & 1 & 1 & 0 & 1 & 1 & 2
\end{array}
\!\!\right]
\end{array}
}
\item[\tt >] \tt{RL := [T[1]*T[2]\textasciicircum 7*T[3]\textasciicircum 8
+T[4]*T[5]\textasciicircum 7*T[6]\textasciicircum 8+T[7]\textasciicircum 8];}
\out{RL := \left[T_1T_2^7T_3^8+T_4T_5^7T_6^8+T_7^8\right]}
\item[\tt >] \tt{R $\sei$ createGR(RL,vars(9),[Q]);}
\out{R\sei GR(9, 1, [2, []])}
\item[\tt >] \tt{X $\sei$ createMDS(R,[1,3]);}
\out{X \sei MDS(9,1,6,[2,[]])}
\item[\tt >] \tt{MDSisfact(X);}
\out{true}
\item[\tt >] \tt{MDSisquasismooth(X);}
\out{false}
\end{enumerate}
\endgroup
\noindent
Since ${\rm Cox}(X)$ is a complete intersection, we may use the formula
presented in~\cite{ADHL} to compute the canonical class of~$X$:
we obtain $\mathcal{K}_X = [4,0] \in \ZZ^2$.
\vspace{0.1cm}
\begingroup
\footnotesize
\setlength{\arraycolsep}{4pt}
\begin{enumerate}[leftmargin=3em]
\item[\tt >] \tt{fujitaBPF(X,[4,0]);}
\out{false}
\item[\tt >] \tt{isBasePointFree(X,[1,3]);}
\out{true}
\end{enumerate}
\endgroup
\noindent 
Note that Algorithm~\ref{algo:fujitabpf} returns false, i.e.~$X$ 
does not fulfill Fujita's base point free conjecture.
To obtain this result the algorithm performs the following steps:
\begin{itemize}
\item 
First Algorithm~\ref{algo:genintmon} is used to compute the generators 
$[0, 1]$ and~$[1,2]$ of $\BPF(X)\subseteq\ZZ^3$.
\item
Then Algorithm~\ref{algo:pointcondid} computes the point
$C \sei[0,0]\in c(\tilde{S}/S)$.  
\item 
The faces of $\cone(S)$ are given by
$F_1\sei\cone([1,2])$, 
$F_2\sei\cone([0,1])$.
The algorithms then computes $\alpha_1,\alpha_2$
such that $-\mathcal{K}_X+C=[-4,0]$ defines a point
in~$\tau_i^{\alpha_i}$. 
We obtain $\alpha_1=8$, $\alpha_2=-4$
and $\nu=8$. Note that $\alpha_2=-4$ is just
the first coordinate of $-\mathcal{K}_X+C$.
\item Then the algorithm performs the following steps:
\begin{itemize}
\item Since we have $\dim(X)+1=7\le m \le 7=\nu-1$, the algorithm
only needs to test the case~$m=7$.
\begin{itemize}
\item For $i=1$ we have $\lfloor\frac{\alpha_1 -1}{7}\rfloor =1$,
i.e.~only the case $k=1$ needs to be considered.
The algorithm yields  $ Gp_i^k= \{[1,3]\}$.
\item 
Now Algorithm~\ref{algo:inmon} is used to test whether 
$\mathcal{K}_X+m \, Gp_i^k \subseteq S$ holds.
We have $\mathcal{K}_X+7 \,[1,3] = [11,21]$
which is not contained in~$\cone(S)$. 
Thus Algorithm~\ref{algo:inmon} returns false.
\end{itemize}
\end{itemize}
\item
Hence the algorithm fujitaBPF returns \textit{false}.
\end{itemize}
\noindent
Note that the $\mathcal{K}_X+7 \,[1,3] = [11,21]$ is  
not semiample and thus not nef. 
Maeda proved in~\cite[Proposition 2.1]{Maeda} that
$\mathcal{K}_X + m \mathcal{L}$ is nef
for all~$m \geq \dim(X)+1$ and for all 
$\mathcal{L} \in \Ample(X)\cap \Pic(X)$ if~$X$ is an  
irreducible normal projective variety with
at most log terminal singularities.
Nevertheless, this example does not contradict the result 
of Maeda since~$X$ is not log terminal:
To see this, one can look at the relevant face~$\gamma_{134}$
and the corresponding affine variety
$
X_{\gamma_{134}}  \sei 
p_X(\overline{X}_{\gamma_{134}})
$,
where $p_X\colon \overline{X}^{\rm ss}(u) \to X$,
$u\sei[1,3]$,
denotes the good quotient with respect to the
$\Spec(\KK[\Cl(X)])$-action on~$\overline{X} \sei \Spec({\rm Cox}(X))$
and where we set
$
\overline{X}_{\gamma_{134}}  \sei 
\overline{X}_{T_1T_3T_4} 
$
as in~\cite[Construction 3.2.1.3.]{ADHL}.
By~\cite{ABHW},~$X_{\gamma_{134}}$ is log terminal only if 
the exponents of different monomials 
are platonic triples. Since this is not the case,
we conclude that~$X$ is not log terminal.

Observe that the base point free monoid
$\BPF(X)\subseteq\ZZ^2$ is saturated and thus the
ample class~$[1,3]$ is base point free.
Although $\mathcal{K}_X+7 \,[1,3] = [11,21]$
is not base point free on~$X$, a result 
of~\cite{LaPaSo} implies 
that~$\mathcal{K}_X+7 \,[1,3] = [11,21]$
is very ample and thus base point free on~$X^{\rm reg}$.

For a geometric description of~$X$, note that 
$X$ admits an elementary contraction~$\varphi\colon X\to \PP_4$ 
of fiber type in the sense of~\cite{Ca}
with fibers isomorphic to a hypersurface
of degree eight in~$\PP_3$. To be precise we have
$\varphi^{-1}(a)\cong V_{\PP_3}(a_1a_2^7T_0^8+a_3a_4^7T_1^8+T_2^8)$
where $a=[a_1,\ldots,a_5]\in\PP_4$ denotes a point of~$\PP_4$
in homogeneous coordinates and where $T_0,T_1,T_2,T_3$ denote the
coordinates of~${\rm Cox}(\PP_3)$.
In addition,~$X$ admits a closed embedding~$X\to Y$ into
the projectivized split vector~bundle
$$
Y \ = \
\PP\big(
\mathcal{O}_{\PP_4} 
\oplus
\mathcal{O}_{\PP_4} 
\oplus
\mathcal{O}_{\PP_4}(1)
\oplus
\mathcal{O}_{\PP_4}(2) 
\big)\, .
$$
\end{example}

We now turn to the proof of Algorithm~\ref{algo:fujitabpf}.

\begin{lemma}\label{boundm}
In the setting of \ref{set:algo}, the following are equivalent:
\begin{enumerate}
\item $\mathcal{K}_X+m \mathcal{L} \in S$ holds for all 
$m \geq \dim(X)+1$ and for all 
ample Cartier divisor classes $\mathcal{L}$, i.e.~$X$ fulfills 
Fujita's base point free conjecture.
\item $\mathcal{K}_X+m\mathcal{L} \in S$ holds for all 
$\nu -1 \ge m \geq \dim(X)+1$ and for all 
ample Cartier divisor classes $\mathcal{L}$.
\end{enumerate}
\end{lemma}

\begin{proof}
Only implication ``(ii)$\Rightarrow$(i)'' needs to be proven. 
Let $m \ge \dim(X)+1$. If~$m \le \nu-1$ holds, then 
$\mathcal{K}_X+m \mathcal{L} \in S$ follows by~(ii).
Now assume that $m \ge \nu$ holds.
Note that since~$\mathcal{L}$ defines a point in the relative
interior of~$\sigma$ for all~$1\le i\le r$,
the multiple $m\mathcal{L}^0\otimes 1$ 
is contained in a facet parallel~$\tau_i^{\beta_i}$ with
$\beta_i\ge m \ge \nu$. Thus by definition of~$\nu$ as 
maximum over
all integers~$\alpha_i$ with 
$(-\mathcal{K}_X^0+C^0) \otimes 1 \in \tau_i^{\alpha_i}$,
we obtain  
$$
m \mathcal{L} \otimes 1 
\ \in \ 
((-\mathcal{K}_X+C) \otimes 1) \ + \ \cone(S)\,.
$$
Thus, $\mathcal{K}_X + m \mathcal{L}$
defines a point in $(C \otimes 1)  +  \cone(S).$
Since $C$ is an element of the conductor ideal~$c(\tilde{S}/S)$ 
of~$S\subseteq K$, we conclude
$
\mathcal{K}_X+ m \mathcal{L} 
 \in  
S$.
\end{proof}

\begin{lemma}\label{boundL1}
In the setting of \ref{set:algo}, the following are equivalent
for $m \in \{\dim(X)+1, \ldots, \nu -1\}:$
\begin{enumerate}
\item 
$\mathcal{K}_X+m \mathcal{L} \in S$ holds for all 
ample Cartier divisor classes $\mathcal{L}$.
\item 
For all $1\le i\le r$
and for all $1\le k \le \lfloor \frac{\alpha_i-1}{m} \rfloor$,
where~$\lfloor\cdot\rfloor$ denotes the floor function,
we have $\mathcal{K}_X+m \mathcal{L}\in S$ for all
$\mathcal{L} \in  \iota_0^{-1} (\tau_i^k \cap \sigma^\circ) \times K^{\rm tor}$.
\end{enumerate}
\end{lemma}

\begin{proof}
Only implication ``(ii)$\Rightarrow$(i)'' needs to be proven. 
Consider an ample Cartier divisor class $\mathcal{L}$, i.e.
$$\mathcal{L} \ \in \ 
\iota_0^{-1}(\sigma^\circ) \times K^{\rm tor}$$
holds. Let $\beta_1,\ldots,\beta_r \in \ZZ_{> 0}$ such 
that~$\mathcal{L}^0 \otimes 1 \in \tau_i^{\beta_i}$ holds.
If~$\beta_i \le \lfloor \frac{\alpha_i-1}{m} \rfloor$ holds
for some $1\le i\le r$, then 
$\mathcal{K}_X+m \mathcal{L} \in S$ follows by~(ii).
Now assume that~$\beta_i > \lfloor \frac{\alpha_i-1}{m} \rfloor$
holds for all $1\le i\le r$.
We obtain
$m\beta_i \ge \alpha_i$ for all $1\le i\le r$.
Recall that
$(-\mathcal{K}_X^0+C^0) \otimes 1 \in \tau_i^{\alpha_i}$ holds for all~$1\le i\le r$.
Thus~$m\beta_i \ge \alpha_i$ for all $1\le i\le r$ shows that
$$
m\mathcal{L}\otimes 1
\ \in \
((-\mathcal{K}_X+C) \otimes 1 ) \ +\  \cone(S) 
$$
holds. Thus, $\mathcal{K}_X + m \mathcal{L}$
defines a point in $(C \otimes 1)  +  \cone(S).$
Since $C$ is an element of the conductor ideal~$c(\tilde{S}/S)$ 
of~$S\subseteq K$, we conclude
$
\mathcal{K}_X+ m \mathcal{L} 
 \in  
S$.
\end{proof}

\begin{lemma}\label{boundL2}
Recall that in the setting of \ref{set:algo}, we defined 
by~$m_1,\ldots,m_{n_i} \in S$ those elements such that 
$m_j^0$ is minimal with the property that $m_j^0 \otimes 1$ 
is contained in a ray of~$F_i$.
Let $1\le i\le r, \;1\le k \le \lfloor \frac{\alpha_i-1}{m} \rfloor$
and consider an ample Cartier divisor class
$\mathcal{L} \in  \iota_0^{-1} (\tau_i^k \cap \sigma^\circ) \times K^{\rm tor}$.
Then there are $y \in Gp_i^k$ and $a_j \in \ZZ_{\ge 0}$
such that we have
$$
\mathcal{L}
\ = \
y + \sum_{j=1}^{n_i} a_j m_j \, .
$$
\end{lemma}

\begin{proof}
Observe that $\sigma \cap \tau_i^k = 
{\rm conv}(p_1^k,\ldots,p_{t_i}^k) +\cone(G_i)$ holds.
Hence there are rational numbers 
$a_j, \, b_\ell \in \QQ_{\ge 0}, \, \sum_{j=1}^{t_i} a_j =1$, 
such that the free part of~$\mathcal{L}$ is given~by
$$
\mathcal{L}^0
\ = \
\sum_{j=1}^{t_i} a_j p_j^k +\sum_{\ell=1}^{n_i} b_\ell m_\ell^0
\, .
$$
We obtain $\mathcal{L} = y + 
\sum_{\ell=1}^{n_i} \lfloor b_\ell\rfloor \, m_\ell \; ~(\ref{boundL2}.1)$, 
where~$\lfloor\cdot\rfloor$ denotes the floor function and
where the free and the torsion part of $y$ are defined as
$$
y^0
\ \sei \ 
\sum_{j=1}^{t_i} a_j p_j^k
\ + \ 
\sum_{\ell=1}^{n_i} (b_\ell-\lfloor b_\ell\rfloor) \, m_\ell^0\, , 
\qquad
y^{\rm tor}
\ \sei \ 
\mathcal{L}^{\rm tor} 
\ - \ 
\sum_{\ell=1}^{n_i} \lfloor b_\ell\rfloor \, m_\ell^{\rm tor} \, .
$$
Note that~$y$ is an element of~$K$
 since we have $y=\mathcal{L}-\sum_{\ell=1}^{n_i} \lfloor b_\ell\rfloor \, m_\ell$, where
 $\mathcal{L}$ as well as the $m_\ell$, $1\le \ell \le n_i$, are elements of~$K$. 
If $y^0 \otimes 1 \in \sigma^\circ$ holds,~$(\ref{boundL2}.1)$ is 
the required representation of $\mathcal{L}$.
Now consider the case where $y^0 \otimes 1$ is not contained in~$\sigma^\circ$.
This means that $y^0\otimes 1 \in \left( \conv(p_1^k,\ldots, p_{t_i}^k) \setminus \sigma^\circ \right)$ holds.
Since $\mathcal{L}^0\otimes 1$ is contained in $\sigma^\circ$,
there is $1\le\ell\le n_i$ with $\lfloor b_\ell \rfloor \neq 0$.
Without loss of generality we assume that 
$\lfloor b_1 \rfloor,\ldots,\lfloor b_{\ell_0} \rfloor > 0$
and 
$\lfloor b_{\ell_0 +1} \rfloor=\ldots=\lfloor b_{\ell_{n_i}} \rfloor
=0$ hold for some $1\le\ell_0\le n_i$.
Then we have
$$
\mathcal{L}  = y^\prime + 
\sum_{j=1}^{\ell_0} \, ( \lfloor b_j\rfloor -1 ) \; m_j \,   \ (\ref{boundL2}.2) \, ,
\quad \text{ where } \quad 
y^\prime \ \sei \  y + \sum_{j=1}^{\ell_0} m_j 
$$
holds. In order to show that formula~$(\ref{boundL2}.2)$ is
the required representation of~$\mathcal{L}$,
it remains to prove that $y^\prime \in G_{p_i}^k$ holds.
Note that $y^\prime \in K$ holds since $y$ is an element of~$K$.
In addition, since
$y^0\otimes 1 \in \left( \conv(p_1^k,\ldots, p_{t_i}^k) \setminus \sigma^\circ \right)$
holds,~$y^\prime$ defines a point in $\conv(p_1^k,\ldots, p_{t_i}^k)+G_i$.
It remains to show that~$y^\prime$ defines a point in the relative interior
of~$\sigma$.
Recall that $\sum_{j=1}^\ell m_j^0\otimes 1$ is contained in the facet~$F_i$.
Furthermore, since we are in the case $y^0\otimes 1\notin\sigma^\circ$,
the point $y^0\otimes 1$ lies in a facet~$F_y$ of $\sigma$.
Since~$k\ge 1$ and~$y\in\iota_0^{-1}(\tau_i^k)$ hold,
we conclude that $y^0\otimes 1$ is not contained in~$F_i$,
i.e.~there is no face~$\kappa\preceq\sigma$
with~$y^0\otimes 1\in  \kappa$ and~$\sum_{j=1}^\ell m_j^0\otimes 1\in \kappa$.
Thus the sum $y^0+  \sum_{j=1}^\ell m_j^0$ defines a point 
in the relative interior of~$\sigma$.
As argued above, this shows that~$y^\prime$ is an element of~$G_{p_i}^k$,
which completes the proof.
\end{proof}

\begin{lemma}\label{boundL3}
In the setting of \ref{set:algo}, let $\dim(X)+1\le m\le\nu -1$, \, $1\le i\le r$
and $1\le k \le \lfloor \frac{\alpha_i-1}{m} \rfloor$.
Then the following are equivalent:
\begin{enumerate}
\item 
$\mathcal{K}_X+m \mathcal{L} \in S$ holds for all 
$\mathcal{L} \in  \iota_0^{-1} \left(\tau_i^k \cap \sigma^\circ\right) \times K^{\rm tor}$.
\item 
$\mathcal{K}_X+m \mathcal{L} \in S$ holds for all 
$\mathcal{L} \in Gp_i^k$.
\end{enumerate}
\end{lemma}

\begin{proof}
Since $Gp_i^k  \subseteq \iota_0^{-1} \left(\tau_i^k \cap
\sigma^\circ\right) \times K^{\rm tor}$ holds, 
only implication ``(ii)$\Rightarrow$(i)'' needs to be proven. 
Note that this is an immediate consequence of Lemma~\ref{boundL2}.
\end{proof}

\begin{proof}[Proof of Algorithm~\ref{algo:fujitabpf}]
We need to show that $X$ fulfills Fujita's base point free
conjecture 
if and only if the above algorithm returns \textit{true}.
This can be seen as follows:
if $X$ is not Gorenstein, then $\mathcal{K}_X+m\mathcal{L}$ 
is not a Cartier divisor class; in particular, it is not
base point free.
Now assume that $X$ is Gorenstein.
Since the embedded monoid
$\BPF(X)\subseteq\Pic(X)$ is spanning, we can apply
Algorithm~\ref{algo:pointcondid} and compute a point of its 
conductor ideal.
Lemma~\ref{boundm} shows that we can bound $m$ by $\nu-1$;
Lemmata~\ref{boundL1} and~\ref{boundL3} prove that the 
sets $Gp_i^k, \;1\le i\le r, \; 1\le k \le 
\lfloor\frac{\alpha_i -1}{m} \rfloor$, serve as validations 
sets of Cartier divisor classes.
\end{proof}


\end{document}